\documentclass[11pt,longtable,letterpaper,oneside,reqno]{amsart}
\usepackage{anysize,placeins,enumitem}
\usepackage[toc,page]{appendix}
\usepackage{booktabs}
\usepackage{caption}
\usepackage{cases}
\usepackage{color}
\usepackage{epsfig}
\usepackage{float,mathtools}
\usepackage[stable]{footmisc}
\usepackage{graphicx}
\usepackage{hyperref}
\usepackage{latexsym}
\usepackage{longtable,tabularx,tabulary}
\usepackage{mathrsfs}
\usepackage{multirow}
\usepackage{textcomp} 
\usepackage{textgreek}
\usepackage{todonotes}
\usepackage{titletoc}
\usepackage{xcolor}

 \DeclareMathOperator{\Gal}{Gal}
 \DeclareMathOperator{\Li}{Li}

\newcommand{\Lo}{{\mathscr{L}}}
\newcommand{\cQ}{{\mathscr{Q}}}

\newcommand{\Q}{{\mathbb{Q}}}

\renewcommand{\Re}{{\mathfrak{Re}}}
\renewcommand{\Im}{{\mathfrak{Im}}}

\newcommand{\s}{\sigma}
\newtheorem{theorem}{Theorem}
\newtheorem{corollary}{Corollary}[theorem]
\newtheorem{lemma}[theorem]{Lemma}

\newtheorem{proposition}[theorem]{Proposition}
\theoremstyle{remark}
\newtheorem*{remark}{Remark}
\newtheorem*{claim}{Claim}
\newtheorem{rmk}[theorem]{Remark}
\numberwithin{equation}{section}
\reversemarginpar 

\begin{document}

\title[Primes in the Chebotarev density theorem for all number fields]{Primes in the Chebotarev density theorem\\for all number fields\\
\textnormal{(With an Appendix by  Andrew Fiori)}}

\author[H. Kadiri]{Habiba Kadiri}
\address{Department of Mathematics and Computer Science\\
University of Lethbridge\\
4401 University Drive\\
Lethbridge, Alberta\\
T1K 3M4 Canada}
\email{habiba.kadiri@uleth.ca}
\email{andrew.fiori@uleth.ca}

\author[P.-J. Wong]{Peng-Jie Wong}
\address{National Center for Theoretical Sciences\\
No. 1, Sec. 4, Roosevelt Rd., Taipei City, Taiwan}
\email{pengjie.wong@ncts.tw}

\thanks{This research was partially supported by the NSERC Discovery grant RGPIN-2020-06731 of H.K. 
P.J.W. is currently an NCTS postdoctoral fellow; he was supported by a PIMS postdoctoral fellowship and the University of Lethbridge during part of this research.}

\begin{abstract}
Let $L/K$ be any Galois extension of number fields $L\not=\Q$, and let $C$ be a conjugacy class in ${\rm Gal}(L/K)$.  We show that there exists a degree-one unramified prime  $\mathfrak{p}$ of $K$ such that $\sigma_{\mathfrak{p}}=C$ and $N \mathfrak{p} \le d_{L}^{B}$ with $B= 310$. This  improves upon $B=12\,577$ obtained by Ahn and Kwon by making the argument of Lagarias, Montgomery, and Odlyzko explicit. Our improvements come from using the following key ideas: new weights, an optimized analysis on the location of the potential exceptional zeros for $\zeta_L(s)$, and a new version of Tur\'an's power sum method which gives a stronger Deuring-Heilbronn phenomenon for $\zeta_L(s)$. We also use Fiori's numerical verification for number fields up to a certain discriminant height. Other results include a lower bound for the number of unramified primes  $\mathfrak{p}$ of $K$ such that $\sigma_{\mathfrak{p}}=C$.
\end{abstract}



\maketitle

\section{Introduction}
Let $L/K$ be a Galois extension of number fields with Galois group $G$, and 
let  $C$ be a conjugacy class in $G$.  The celebrated Chebotarev density theorem asserts that 
$$\# \{ \mathfrak{p} \subset \mathcal{O}_K \ | \ N \mathfrak{p} \le x,  \ \text{$\mathfrak{p}$ is an unramified prime with $\sigma_{\mathfrak{p}}=C$}  \} \sim  
  \frac{|C|}{|G|} \Li(x)
$$
as $x \to \infty$, where $ \Li(x)$ is the usual logarithmic integral,  $\mathcal{O}_K$ is the ring of integers of $K$,  $N=N_{K/\mathbb{Q}}$ is the absolute norm of $K$, and  $\sigma_{\mathfrak{p}}$ denotes the Artin symbol at $ \mathfrak{p}$.

As the  Chebotarev density theorem generalizes Dirichlet's theorem on primes in arithmetic progressions, a natural question of finding  the least (unramified) prime $ \mathfrak{p}$ with $\sigma_{\mathfrak{p}}=C$ then arises from Linnik's famous theorem on the least prime in an arithmetic progression. Under the generalized Riemann hypothesis for the Dedekind zeta function $\zeta_L(s)$ of $L$, Lagarias and Odlyzko  
  \cite{LO} showed that  $N \mathfrak{p} \ll (\log d_L)^2$, where  $d_L$ denotes the absolute discriminant of $L$ (cf. \cite{BS}). 
 
In  \cite{LMO}, Lagarias, Montgomery, and Odlyzko  proved,  unconditionally, that if $L \ne \mathbb{Q}$, then there is a constant $B>0$ such that there is an unramified prime  $\mathfrak{p}$ of $K$ with $\sigma_{\mathfrak{p}}=C$ and
$N \mathfrak{p} \le d_{L}^{B}$. Recently, Zaman \cite{Z17} established that $B=40$ is valid when $d_L$ is sufficiently large. This is improved by  Ng and  the authors in \cite{KNW}, who showed  that   $B=16$ is admissible for sufficiently large $d_L$. Also, Ahn and Kwon \cite{AK} showed that $B=12\,577$ 
is valid for \emph{all} number fields.
The main result of this article is the following theorem. 
\begin{theorem}  \label{mainthm}
Let $L/K$ be a Galois extension of number fields with Galois group $G$, and let  $C$ be a conjugacy class in $G$.  If $L\neq\Bbb{Q}$, then there exists an unramified prime  $\mathfrak{p}$ of $K$, of degree one,  such that $\sigma_{\mathfrak{p}}=C$ and $N \mathfrak{p} \le d_{L}^{B}$ with $B= 310$.
\end{theorem}  

Throughout this article, we let $n_L\ge n_0\ge 2$ and $d_L\ge d_0\ge 3$.
In Appendix \ref{Appendix}, numerical verifications of the bound for the least prime, with $B\le 1.7712$, are done for $n_L=n_0$ with $2\le n_0 \le 20$ and $d_L \le d_0$, as well as for $n_L\ge n_0\ge 21$ and $d_L \le d_0 = 10^{n_L}$ (see Table \ref{Table-Fiori}). 

In Section \ref{leastprime}, we prove the least prime result for the remaining $(n_0,d_0)$ (see Table \ref{table}).
The proof starts by establishing an explicit inequality (see Section \ref{expl-ineq}) between a weighted sum over primes detecting the least one in a given class to the zeros of the Dedekind zeta function $\zeta_L(s)$. The size of the least prime then depends on how close to the vertical line $\Re s=1$ these zeros are. 
We can establish some zero-free regions:
there exist absolute constants $A,A'>0$ such that $\zeta_L(s)$ has no zeros in the region 
$$
  \Re(s) \ge 1 - \frac{1}{A\log d_L + A' n_L\log (|\Im (s)|+2) }, 
$$
with the exception of at most one real zero $\beta_1$.
It has been proven in \cite[Theorem 1.1]{Ka12} that for $ |\Im (s)|\le 1$, $A=12.74$, and $A'=0$  assuming $d_L$ is sufficiently large.
Recently, Lee improved this to $A=12.44$ in \cite[Theorem 2]{Lee20}.
Also, this was made explicit in \cite[Proposition 6.1]{AK} with $A=A'=29.57$.
We will also need a region with no non-exceptional zeros of the form $ \Re s  \ge 1 - \frac{1}{R_1\log d_L}$ and $|\Im s| \le \frac{1}{R_1\log d_L}$ with some admissible $R_1\ge 1.24$ as discussed in Section \ref{Z-F-R}. 

We can refine the statement of Theorem \ref{mainthm} depending on whether $\zeta_L(s)$ possesses an exceptional real zero $\beta_1$ or not. For instance, if $\zeta_L(s)$ has no exceptional zero, then Theorem \ref{mainthm} is valid with $B=10.5$.
In the other case, we require a careful study depending on how close the exceptional zero $\beta_1$ is to $1$, and we adapt our choice of weight (see Section \ref{weight}) to detect the least prime accordingly. In particular, we determine that it is useful to investigate the supplementary case where $\beta_1$ is at a ``very small'' distance, namely
$(\log d_L)^{-2} \ll (1-\beta_1) \ll (\log d_L)^{-1.15}$. 

It is known as a Deuring-Heilbronn phenomenon that the closer $\beta_1$ is to $1$, the further left other zeros are ``repulsed". We now describe the stronger zero-repulsion theorems we use instead of the one established in  \cite[Theorem 7.3]{AK}; in Section \ref{D-H}, we prove some versions of \cite[Theorems 1.2 and 1.3]{KNW} which are valid for all number fields. 

\begin{theorem}  \label{thm-DH-all}
Let $L\neq \Bbb{Q}$ be a  number field of degree $n_L$ and with absolute discriminant $d_L$. Assume $\zeta_L(s)$ admits an exceptional real zero $\beta_1$. Let 
$\beta' +i t$ be another (non-trivial) zero of $\zeta_L(s)$, and let $\tau=|t|+2$. Then, for any $\eta\in (0,1]$, there exist $c_1$ and $c_2$ such that
either $\beta' \le 1-\eta$ or
\begin{equation}
\beta' \le  1 -c_2 \frac{ \log \Big( \frac{c_1}{(1-\beta_1) \log (d_L \tau^{n_L})} \Big)}{\log (d_L \tau^{n_L})}.
\end{equation}
In addition, if $|t|\le1$, then there exist $c_1^{\prime}$ and $c_2^{\prime}$ such that
either $\beta' \le 1-\eta$ or
\begin{equation}
     \beta' \le  1 - c_2^{\prime} \frac{ \log \Big( \frac{c_1^{\prime}}{(1-\beta_1) (\log d_L)} \Big)}{\log d_L}.
\end{equation}
Here $(c_1,c_2)$ and $(c_1^{\prime},c_2^{\prime})$ depend on $\eta$, and are respectively defined in \eqref{def-c7-c8}
and \eqref{def-C7-C8}.
\end{theorem}
We also investigate the case of real zeros as this case offers a stronger zero-repulsion phenomenon, as illustrated in the next theorem.
\begin{theorem}  \label{dhall_real}
Assume $\zeta_L(s)$ admits an exceptional real zero $\beta_1$.   Let 
$\beta'$ be another real zero of $\zeta_L(s)$. Then, for any $\eta\in (0,1]$, there exist $c_1^{\prime\prime}$ and $c_2^{\prime\prime}$ such that
either $\beta' \le 1-\eta$ or
\begin{equation}\label{real_bd}
\beta'\le 1- c_2^{\prime\prime} \frac{\log\Big( \frac{c_1^{\prime\prime}}{(1-\beta_1)\log d_L } \Big)}{\log d_L},
\end{equation}
where $(c''_1,c''_2)$, depend on $\eta$, are defined in \eqref{def-c7'-c8'}.
For $\eta=1$, numerical results are listed in Table \ref{c1-c2-c3}.
\end{theorem}
The key tool to prove Theorems \ref{thm-DH-all} and \ref{dhall_real} resides in Theorem \ref{newTuran} where we improve lower bounds for some modified Tur\'an power sums. 
Numerical results for $\eta=1$ are listed in Table \ref{c1-c2-c3}.
For instance, we establish that $c_2=0.04233\approx \frac{1}{23.624}$, $c_2^{\prime}=0.05=\frac1{20}$, and $c''_2=0.1008\approx \frac1{9.921}$ are admissible. 
This enlarges the region described in \cite[Theorem 7.3(2)]{AK}, where $c_2=\frac1{77}$ was obtained.\footnote{The labelling for our  $(c_1,c_2)$ was $(c_7,c_8)$ in \cite[Theorem 7.3(2)]{AK}.} In addition this extends and improves \cite[Theorem 1.2]{Z17} where Zaman proved $c_2^{\prime}=\frac1{35.8}$ for $d_L$ sufficiently large. In \cite{KNW}, the repulsion constant was improved to $\frac1{14.144}$, for $d_L$ sufficiently large.

In Section \ref{real-zeros} we deduce a bound for the possible exceptional zero $\beta_1$  of $\zeta_L(s)$. 
\begin{theorem}  \label{upperbd}
Suppose $L\neq\Bbb{Q}$. If $\zeta_L(s)$ admits a real zero $\beta_1$,  then
\[
   1-\beta_1 \ge d_{L}^{-c_3} \ \text{with}\ c_3= 11.7 .
\]
\end{theorem} 
This improves \cite[Corollary 7.4]{AK} where $c_3 = 114.72...$ was proven admissible. 
We note that for $d_L$ sufficiently large, the bound above was established with $c_3=16.6$ in \cite[Corollary 1.4]{Z17} and $c_3=7.072$ in \cite[Corollary 1.3.1]{KNW}. 

At this point, we have all the key ingredients (also summarized in Section \ref{keys}), and we will complete the proof of Theorem \ref{mainthm} in Section \ref{leastprime}. Here is a qualitative description to illustrate how they impact the final result:
\begin{itemize}[leftmargin=15pt,labelindent=0pt,labelwidth=12pt,itemindent=8pt,topsep=0pt]
\item a refined case analysis (in particular, adding the ``very small'' case) 
allows us to divide Ahn and Kwon's constant of $12\,577$ by a factor of about $7$;
\item the choice of the weight 
and the strength of Deuring-Heilbronn phenomenon allow us to improve the result by a factor of about $3$;
\item finally, the numerical verifications divides the value of $B$ by 2 (going from $620$ without using Table \ref{Table-Fiori}, to the announced $310$). 
\end{itemize}

We let $\tilde{\pi}_{C}(x)$ denote the number of degree-one unramified primes $\mathfrak{p}$ of $K$  such that  $N \mathfrak{p} \le x$ and $\sigma_{\mathfrak{p}}=C$. 
In Section \ref{lower-bound} we apply Theorem \ref{upperbd} to obtain a lower bound for $\tilde{\pi}_{C}(x)$.
\begin{theorem}\label{lwr-bnd}
Let $L/K$ be a Galois extension of number fields with Galois group $G$, and let  $C$ be a conjugacy class in $G$.  If $L\neq \Bbb{Q}$, then for $x\ge \exp(d_L^{c_3})$, we have
$$
\tilde{\pi}_{C}(x)\ge m \frac{|C|}{|G|}\frac{x}{\log x},
$$
where $c_3 = 11.7$ and $m=  0.4899$.
\end{theorem} 
This improves \cite[Theorem 3]{AK2} where Ahn and Kwon obtained instead $c_3=114.72...$ and $m= 0.353$ for all $L\neq \Q$. 
We also note that  Thorner and  Zaman \cite[Theorem 3.1]{TZ17} showed that there are absolute constants $\kappa_2$ and $\kappa_3$ such that
$$
\tilde{\pi}_{C}(x)\gg \frac{1}{d_L^{\kappa_2}} \frac{|C|}{|G|}\Li(x)
$$
for $x\ge d_L^{\kappa_3}$ and $d_L$ sufficiently large. 
While our range for $x$ is more restricted, we are able to obtain a lower bound independent of $L$.

\subsection*{Notation}
We recall that $L$ is a number field of degree $n_L\ge n_0\ge 2$ and with absolute discriminant $d_L\ge d_0\ge 3$.
We denote
\begin{equation}\label{def-L}
\Lo=\log d_L \ \text{and}\ \Lo_0 = \log d_0.
\end{equation}
Together with Minkowski's bound, we consider 
\begin{equation}\label{assumptions}
n_0 \le n_L \le \frac{2\log d_L}{\log 3}, \ d_L \ge d_0, \ \text{and}\ \log d_L \ge \Lo_0 .
\end{equation}
Let $\tau=|t|+2$. We shall define 
\begin{equation}\label{def-delta}
\delta = \delta_L(\tau)=\frac{\log(\tau^{n_L})}{\log d_L}, 
\end{equation}
and use various bounds for $\delta $ depending on cases. For $T_0\ge0$, we define 
\begin{equation}\label{def-Delta}
\Delta_0(T_0) =  \mathscr{Q}_0 \log(T_0+ 2),
\end{equation}
where
\begin{equation}\label{def_Q}
\mathscr{Q}_0
 = \begin{cases}
 \frac{n_L}{\Lo_0 } & \textrm{ if $2\le n_L\le 20$;}\\
  \frac{1 }{\log 10} & \textrm{ if $ n_L \ge 21$.}
  \end{cases}
\end{equation}
Note that for $|t| \le T_0$, we have
\begin{equation}\label{bnd-delta}
0\le \delta \le \Delta_0(T_0).
\end{equation}

\section{Repelling non-exceptional zeros further left}\label{Z-F-R-D-H}
We recall \cite[Proposition 6.1]{AK}, which asserts that for any number field $L$, there are no zeros for $\zeta_L(s)$ in the region 
\begin{equation}  \label{zfr}
  \Re(s) \ge 1 - \frac{1}{29.57 \left( \Lo + n_L\log (|\Im (s)|+2) \right)},  
\end{equation}
with the exception of at most one real zero $\beta_1$. If the exceptional zero $\beta_1$ exists, then it has to be real and simple.
In addition, by \cite{Ka12} (for $d_L$ sufficiently large) and by \cite[Theorem 1]{AK14} (for all $L\neq \Bbb{Q}$), we know that $\zeta_L(s)$ has at most one zero  in the region
\begin{equation}\label{zfr-R0}
\beta> 1-\frac{1}{R_0\Lo} \text{ and } |\gamma|<\frac{1}{2\Lo},
\end{equation}
for any $R_0\ge 2$. In what follows we let $\beta_1$ denote the  exceptional zero of $\zeta_L(s)$ if it exists. We shall split our considerations into several cases depending on the location (and existence) of  $\beta_1$.

\subsection{Enlarging the region without non-exceptional zeros}\label{Z-F-R}
We assume that there exists $R_1\le R_0$ such that there is no non-trivial zeros $\rho \neq \beta_1$ of $\zeta_L(s)$ in the region 
\begin{equation}\label{zfr-enlarged-R1}
\beta> 1-\frac{1}{R_1\Lo} \text{ and } |\gamma|<\frac{1}{R_1\Lo}.
\end{equation}
We note that one can always take $R_1=2$, but show here that we can take 1.24 and 1.7 as  admissible values for $R_1$ for certain instances:
\begin{proposition}\label{lemma-zfr-enlarged}
Let $L$ be a number field such that $L\not=\Q$, and let $R_0\ge 2.$
If the exceptional zero $\beta_1$ of $\zeta_{L}(s)$ presents in $(1-\frac{1}{R_0\Lo},1)$, then there is no other zeros of $\zeta_L(s)$ in the region 
\begin{equation}\label{zfr-enlarged}
\beta> 1-\frac{1}{1.7\Lo} \text{ and } |\gamma|<\frac{1}{1.7\Lo}.
\end{equation}
Moreover, if $R_0\ge 3.5$, then $1.7$ in \eqref{zfr-enlarged} can be improved to $1.24$.
\end{proposition}

\begin{proof}[Proof of Proposition \ref{lemma-zfr-enlarged}]
 Suppose that the exceptional zero $\beta_1$ presents in $(1-\frac{1}{R_0\Lo},1)$. Let $\beta_2+i\gamma_2$ be another non-trivial zero of $\zeta_L(s)$ such that
\begin{equation*}
\beta_2> 1-\frac{1}{R\Lo} \text{ and } |\gamma_2|<\frac{1}{R\Lo}.
\end{equation*}
Following  \cite[Sec. 2]{AK14}, we assume $R\ge 1.24$ and set
$
\sigma=1+\frac{1}{r \Lo},
$
with $1<\sigma\le 6.2$, and $\frac{1}{5.2\log 3} \le r \le R$. 
As argued in  \cite[pp. 438--439]{AK14}, one has
$$
\Big(\frac{1}{2}\Big(1 -\frac{1}{\sqrt{5}}  \Big) +r  \Big) \Lo \ge \frac{\sigma - \beta_1}{(\sigma-\beta_1)^2} + \frac{\sigma - \beta_2}{(\sigma-\beta_2)^2 +\gamma_2^2} 
 = \frac{1}{\sigma-\beta_1} + \frac{\sigma - \beta_2}{(\sigma-\beta_2)^2 +\gamma_2^2},
$$
and thus
\begin{equation}\label{zfr-const}
\frac{1}{2}\Big(1 -\frac{1}{\sqrt{5}}  \Big) +r \ge \frac{rR_0 }{R_0+r } + \frac{r R(R+r)}{(R+r)^2 +r^2} .
\end{equation}
By a numerical calculation, if $R_0\ge 2$ and $r=0.6$, it can be checked that \eqref{zfr-const} fails for $R\ge 1.7$. Also, if $R_0\ge 3.5$ and $r=0.6$,  \eqref{zfr-const} fails for $R\ge 1.24$. This concludes the proof. 
\end{proof}

\subsection{Deuring-Heilbronn phenomenon for all number fields}\label{D-H}

In this section, we prove Theorems \ref{thm-DH-all} and \ref{dhall_real}, that is quantitative versions of the Deuring-Heilbronn phenomenon for all number fields. To do so, we employ the Tur\'an's power sum method as introduced by Lagarias, Montgomery, and Odlyzko in \cite{LMO} (see also Montgomery's \cite[Theorem 11]{Mon10}). We note that Ahn and Kwon employ \cite[Theorem 4.2]{LMO} to prove their version \cite[Theorem 7.3]{AK} of the Deuring-Heilbronn phenomenon for all number fields.  
The reader can also find a refinement of \cite[Theorem 4.2]{LMO} in Zaman's \cite[Theorem 2.3]{Z17}.
In \cite[Theorem 2.2]{KNW}, the authors, together with Ng, appeal to Harnack's inequality to bring some improvements to these results. The following is the case $s_j=\sum_{n\ge 1} b_n z_n^j$ with $b_n=1$ for all $n$.

\begin{theorem}
\label{newTuran}
Let $\varepsilon>0$ and $z_1\neq 0$.
Assume that for all $n\ge 1$, $z_n$ is a complex number satisfying $|z_n| \le |z_1|$. 
For any $j\in\Bbb N$, set $s_j=\sum_{n\ge 1}z_n^j$ and 
$
M=\sum_{n\ge 1}\frac{|z_n|}{|z_1|+|z_n|}$.
Then there exists $j_0$ with $1\le j_0\le (8+\varepsilon)M$ such that 
\[
\Re(s_{j_0})
\ge  \frac{ \varepsilon }{4(8+\varepsilon)} |z_1|^{j_0}.
\]
\end{theorem}
We propose here a bound for $\Re \frac{\Gamma'}{\Gamma} $ which is the last tool we need to prove Theorem \ref{thm-DH-all}.
It improves \cite[Lemma 5.3]{AK} and will allow for sharper estimates for $\Re \frac{\gamma_L'}{\gamma_L}$ than \cite[Lemma 5.4]{AK}, where $\gamma_L$ is the associated gamma factor of $\zeta
_L(s)$.

\begin{lemma}\label{bnd-gammas}
Let $T_0\ge 0$. For  $s=\sigma +it$ with $\sigma>0$ and  $|t|\le T_0$,
we define
\begin{equation} \label{def-g}
g(\s,T_0) =  \max_{|t|\le T_0}\Big( \log\Big(\frac{\sqrt{\sigma^2 +t^2 }}{|t|+2} \Big) - \frac{\sigma}{ \sigma^2 +t^2}\Big)    +\frac{1}{3 \sigma^2 } -\log 2,
\end{equation}
Then
\begin{equation}\label{bnd-Gamma}
\Re \frac{\Gamma'}{\Gamma}\Big(\frac{s}{2}\Big) \le  \log (|t|+2) + g(\s,T_0).
\end{equation}
\end{lemma}

\begin{proof}
According to \cite[p. 1429]{AK}, 
\begin{equation}\label{bd-Gamma}
\Re \frac{\Gamma'}{\Gamma}(s)
 \le  \log |s| - \frac{\sigma}{2 (\sigma^2 +t^2)}  +\frac{1}{12 \sigma^2 }. 
\end{equation}
Thus, 
\[
\Re \frac{\Gamma'}{\Gamma}\Big(\frac{s}{2}\Big)
\le  \log (|t|+2) + \max_{|t|\le T_0}\Big( \log\Big(\frac{\sqrt{\sigma^2 +t^2 }}{|t|+2} \Big) - \frac{\sigma}{ \sigma^2 +t^2}\Big)    +\frac{1}{3 \sigma^2 } -\log 2
\]
from which \eqref{bnd-Gamma} follows. 
\end{proof}
Now, we are in a position to prove Theorems \ref{thm-DH-all} and \ref{dhall_real}.
For both proofs, we are assuming \eqref{assumptions} for $n_L$ and $d_L$, and we introduce $\varepsilon >0, \sigma\ge2$, and $\eta\in (0,1]$.

\begin{proof}[Proof of Theorem \ref{thm-DH-all}]
We denote $\beta_1$ the exceptional real zero that we assume $\zeta_L(s)$ possesses. We let $\mathcal{S}$ be the set of all non-trivial zeros of $\zeta_L(s)$. 
Let $\eta \in (0,1]$ and $T_0\ge0$. We also denote $\beta'+it$ a non-trivial zero of $\zeta_L(s)$ such that $1-\eta \le \beta'<1$ and $|t|  \le T_0$. 
We shall let $\sigma\ge 2$ and let $z_n$ run over $(\sigma-\rho )^{-2}$ and $(\sigma+it-\rho )^{-2}$ for all $\rho \in\mathcal{S}\backslash\{\beta_1\}$ in order that $|z_n|$ decreases (so $|z_1|\ge|z_2|\ge\cdots$). 
Thus, $ |z_1| \ge \frac{1}{(\sigma-\beta')^{2}}$, and by the inequality $\log(1+x) \le x$, we have 
\begin{equation}\label{bnd-z1}
 |z_1| \ge  \frac{1}{(\sigma-1)^{2}} \frac{(\sigma-1)^{2}}{(\sigma-\beta')^{2}}
 = \frac{1}{(\sigma-1)^{2}}  \exp \Big( -2 \log \Big(\frac{\sigma-\beta'}{\sigma-1}\Big) \Big)\ge 
 \frac{1}{(\sigma-1)^{2}}  \exp \Big( -2  \frac{1-\beta'}{\sigma-1} \Big).
\end{equation}
From \cite[Eq. (7.2) and Lemma 7.1]{AK}, it follows that
\begin{equation}\label{sum-zn}
\Re \sum_{n\ge 1} z_n^{j_0} \le  \frac{4j_0(1-\beta_1)}{(\s-1)^{2j_0+1}}. 
\end{equation}
Applying Theorem \ref{newTuran} for $\varepsilon>0$, there exists $j_0$ with $1\le j_0\le (8+\varepsilon)M$ such that 
\begin{equation}\label{newTuran'}
\Re\sum_{n\ge 1}z_n^{j_0}
\ge  \frac{ \varepsilon }{4(8+\varepsilon)} |z_1|^{j_0}.
\end{equation}
Combining \eqref{bnd-z1}, \eqref{sum-zn}, \eqref{newTuran'}, and the fact $j_0 \le (8+\varepsilon) M$ gives
\begin{equation}\label{bnd-Turan'}
1-\beta' \ge  \frac{\sigma-1}{2(8+\varepsilon) M }\log \Big(  \frac{ \varepsilon  (\s-1)}{16M(8+\varepsilon)^2 (1-\beta_1)} \Big).
\end{equation}
We finalise by appealing to \cite[p. 2294]{KNW} to establish an upper bound for $M$:
\begin{equation}\label{bnd-M} 
M \le  \mathcal{A} S_L(d-1,t),
\end{equation}
where $\mathcal{A}$, $d$, and $S_L(d-1,t)$ are respectively defined in \cite[Eq. (2.6), (2.11), and (2.12)]{KNW}:
\begin{align}
& \label{def-A} 
\mathcal{A}=\mathcal{A}(\s,\eta)= (\s-1+\eta)^2,
\\ &\label{def-d} 
d =d(\s,\eta) = \sqrt{\s^2+\mathcal{A}} =  \sqrt{2\s^2+(1-\eta)^2-2\sigma(1-\eta)} ,
\\& \label{def-SL} 
 S_L(d-1,t) = \sum_{\rho} \Big( \frac{1}{|d-\rho|^2} + \frac{1}{|d+it-\rho|^2} \Big).
\end{align}
Also, \cite[Lemma 2.5 and Eq. (2.10)]{Z17} provides a bound for $ S_L$:
\begin{equation}\label{bnd-SL}
S_L(d-1,t) \le \frac{\log d_L}{d-1} + \frac{G_1(d-1;|t|)}{d-1} r_1+ \frac{G_2(d-1;|t|)}{d-1} (2r_2) + \frac{2}{(d-1)^2}+ \frac{2}{d(d-1)},
\end{equation}
where $r_1$ and $r_2$ are the numbers of real and complex places, respectively, of $L$,  and 
\begin{equation}\label{expression-G1+G2}
 \begin{split}
G_1(d-1;|t|)  r_1+  G_2(d-1;|t|) (2r_2) 
=& \frac{r_1+r_2}2 \frac{\Gamma'}{\Gamma} \Big( \frac{d}{2}\Big) 
+ \frac{r_1+r_2}2  \Re \frac{\Gamma'}{\Gamma} \Big( \frac{d+i|t| }{2}\Big)\\
& +\frac{r_2}2 \frac{\Gamma'}{\Gamma} \Big( \frac{d+1}{2}\Big) 
 + \frac{r_2}2 \Re \frac{\Gamma'}{\Gamma} \Big( \frac{d+1+i|t| }{2}\Big) 
 - n_L \log \pi .
  \end{split}
\end{equation}
Together with the bounds on Gamma from Lemma \ref{bnd-gammas} and the fact that $ \frac{\Gamma'}{\Gamma}(x) $ increases with real values $x$, we obtain
\begin{equation}\label{bd-G1+G2}
 r_1 G_1(d-1;|t|) + 2r_2 G_2(d-1;|t|) 
\le 
 \frac{\log (\tau^{n_L})}2  
+ \frac{ n_L}2 \Big( \frac{\Gamma'}{\Gamma} \Big( \frac{d+1}{2}\Big)  + \max_{x=d,d+1} g(x,T_0) -2 \log \pi \Big).
\end{equation}
We observe that, for $\delta$ defined in \eqref{def-delta}, whenever $\log d_L\ge \Lo_0$,
\begin{align*} 
& \log d_L  = \frac{1}{1+\delta} \log \Big( d_L \tau^{n_L} \Big) , \quad
\log(\tau^{n_L}) =  \frac{\delta}{1+\delta}  \log (d_L \tau^{n_L}) ,
\\
& n_L  \le \frac{\delta}{1+\delta} \frac{ 1}{\log 2} \log (d_L \tau^{n_L}) , \quad
1 \le  \frac{1}{(1+\delta)}\frac{1}{\Lo_0}\log \Big( d_L \tau^{n_L} \Big).
\end{align*}
We combine these  with \eqref{bd-G1+G2} so that \eqref{bnd-SL} becomes
\begin{equation}
S_L(d-1,t) 
\le \frac{\log (d_L \tau^{n_L})}{d-1} \Big(
 \frac{b_1(d)+b_2(d,T_0)\delta }{1+\delta} 
 \Big),
\end{equation}
where we denote
\begin{align}
& \label{def-b1}
b_1(d) =  1+ \frac{2}{\Lo_0(d-1)}+ \frac{2}{\Lo_0d} ,
\\ & \label{def-b2}
b_2(d,T_0) =  \frac{1}{2} 
+ \frac1{2\log2} \max\Big\{  \frac{\Gamma'}{\Gamma} \Big( \frac{d+1}{2}\Big)  +  \max\{g(d,T_0),g(d+1,T_0) \} - 2\log \pi, 0   \Big\}.
\end{align}
Now, setting
\begin{equation}
\label{def-mathcalB}
\mathcal{B} =\mathcal{B}(d,T_0)  =
\frac{1}{d-1} \max_{\delta\in [0,\Delta_0(T_0)]} \Big( \frac{b_1(d)+b_2(d,T_0)\delta }{1+\delta}  \Big),
\end{equation}
we see that \eqref{bnd-M} becomes
$M
\le \mathcal{A} \mathcal{B}\log (d_L \tau^{n_L}).
$
From \eqref{bnd-Turan'} and \eqref{bnd-M}, it follows that
\[
1-\beta'
\ge  \frac{\sigma-1}{2 (8+\varepsilon) \mathcal{A} \mathcal{B} }
\frac{\log \Big(  \frac{ \varepsilon (\s-1)}{16(8+\varepsilon)^2  \mathcal{A} \mathcal{B}}\frac1{ (1-\beta_1) \log (d_L \tau^{n_L})}\Big)}{\log (d_L \tau^{n_L})}
=c_2 \frac{\log\Big( \frac{c_1}{(1-\beta_1)\log (d_L \tau^{n_L})} \Big)}{\log (d_L \tau^{n_L})},
\]
with
\begin{equation}
 \label{def-c7-c8}
c_1
=  \frac{ \varepsilon }{8(8+\varepsilon)}c_2
\ \text{ and }\  
c_2
=\frac{\sigma-1}{2(8+\varepsilon)\mathcal{A}\mathcal{B}}. 
\end{equation}
Now, taking $T_0=\infty$, we establish the first part of Theorems \ref{thm-DH-all}. 
We note that one may bound $g(\s,\infty)$, trivially, by
$$
g(\s,\infty) \le  \frac{1}{2} \log\Big(\frac{\sigma^2}{4} +1 \Big) +\frac{1}{3 \sigma^2 } -\log 2.
$$
We will use this bound to control $\mathcal{B}(d,\infty)$ and calculate $c_1$ and $c_2$ for the theorem.

On the other hand, for $|t|\le 1$ (so for $T_0=1$), we may directly use \eqref{bd-Gamma}
to deduce
\begin{equation}\label{bnd-gammas-t<1}
\max_{x=d,d+1}\Re \frac{\Gamma'}{\Gamma}\Big( \frac{x+it }{2}\Big)
\le  \max_{x=d,d+1} \Big(\log \sqrt{x^2 +1}  - \frac{x}{x^2 + 1}  +\frac{1}{3 x^2 }\Big)-  \log 2.
\end{equation}
Hence, by \eqref{expression-G1+G2}, we have
$
 G_1(d-1;t)  r_1+  G_2(d-1;t) (2r_2) 
 \le n_L \mathcal{H}(d) ,
$
where we denote
\begin{equation}\label{def-H}
\mathcal{H}(d)= \frac{1}{2} \frac{\Gamma'}{\Gamma} \Big( \frac{d+1}{2}\Big) + \frac{1}{2} \max_{x=d,d+1} \Big(\log \sqrt{x^2 +1} - \frac{x}{x^2 + 1}  +\frac{1}{3 x^2 }\Big) -   \frac{\log 2 + 2 \log \pi}{2},
\end{equation}
Thus, whenever $n_L\ge n_0$ and $\frac{n_L}{\log d_L}\le \mathscr{Q}_0$, we obtain
\begin{align*}
S_L(d-1,t) 
\le & \frac{1}{d-1}  \Big( \log d_L + n_L \mathcal{H}(d)  + \frac{4d-2}{d(d-1)} \Big) \\
\le & \frac{\log d_L}{d-1}  \Big( 1 +  \mathscr{Q}_0 \max\Big\{ \mathcal{H}(d) + \frac{4d-2}{d(d-1)n_0 }, 0 \Big\}\Big),
\end{align*}
with $\mathscr{Q}_0$ as defined in \eqref{def_Q}. 
Now, setting
\begin{equation}
\label{def-mathcalB<1}
\mathcal{B}'=\mathcal{B}'(d)  = \frac{ 1 +  \mathscr{Q}_0\max\Big\{0, \mathcal{H}(d) + \frac{4d-2}{d(d-1)n_0 } \Big\} }{d-1},
\end{equation}
\eqref{bnd-M} gives
$
M  \le \mathcal{A} \mathcal{B}' \log d_L.
$
Applying \eqref{bnd-Turan'}, we arrive at
$
1-\beta' 
\ge c_2^{\prime} \frac{\log\Big( \frac{c_1^{\prime}}{(1-\beta_1)\log d_L} \Big)}{\log d_L},
$
where 
\begin{equation}\label{def-C7-C8}
c_1^{\prime}= c_1^{\prime}(\varepsilon,\sigma,\eta)= \frac{ \varepsilon }{8(8+\varepsilon)}c_2^{\prime}\ \text{ and }\ 
c_2^{\prime}=c_2^{\prime}(\varepsilon,\sigma,\eta) = \frac{\sigma-1}{2(8+\varepsilon)\mathcal{A}\mathcal{B}'} . 
\end{equation}
This completes the proof of Theorem \ref{thm-DH-all}
\end{proof}
We shall now prove the version of the Deuring-Heilbronn phenomenon that \emph{only} concerns the location of real zeros of $\zeta_L(s)$. We adapt the previous proof to the case $t = 0$.
\begin{proof}[Proof of Theorem \ref{dhall_real}]
The proof repeats most of the argument from the previous proof, so we shall keep using the notation with $t = 0$, and with the sequence 
$z_n$ running over $(\sigma-\rho )^{-2}$ for all the non-trivial zeros $\rho \in \mathcal{S}\backslash\{\beta_1\}$, with $z_n$ satisfying $|z_1|\ge|z_2|\ge\cdots$. 
It follows from \cite[Proof of Lemma 2.7]{KNW} that $2 M\le \mathcal{A} S_L(d-1,0)$. We observe that
\[
\frac{S_L(d-1,0)}{2} 
\le \frac{1}{d-1}  \Big( \frac{\log d_L}{2} + \frac{n_L}{2}\Big( \frac{\Gamma'}{\Gamma} \Big( \frac{d+1}{2}\Big) - \log \pi \Big)   + \frac{2d-1}{d(d-1)} \Big) .
\]
Thus, setting
\begin{equation}\label{def-mathcalB''}
\mathcal{B}''=\mathcal{B}''(d) 
=\frac{1}{d-1}\Big( \frac{1}{2}  + \mathscr{Q}_0\max \Big\{ 0, \frac{1}{2} \frac{\Gamma'}{\Gamma} \Big(\frac{d+1}{2}\Big) - \frac{\log \pi}2 +\frac{2d-1}{d(d-1)n_0} \Big\} \Big) ,
\end{equation}
we have 
$
M
\le \mathcal{A}\frac{S_L(d-1,0)}{2} 
 \le \mathcal{A}  \mathcal{B}'' (\log d_L)
$
whenever $n_L\ge n_0$, and $\frac{n_L}{\log d_L}\le \mathscr{Q}_0$. 
Hence, arguing as before, we have
$
1-\beta'\ge c_2^{\prime\prime} \frac{\log\Big( \frac{c_1^{\prime\prime}}{(1-\beta_1)\log d_L } \Big)}{\log d_L},
$
where $c_1^{\prime\prime},c_2^{\prime\prime}$ depend on $(8+\varepsilon),\sigma,\eta$ and are given by
\begin{equation}\label{def-c7'-c8'}
c_1^{\prime\prime}=  \frac{ \varepsilon }{8(8+\varepsilon)}c_2^{\prime\prime}\ \text{ and }\ c_2^{\prime\prime}=\frac{\sigma-1}{2(8+\varepsilon)\mathcal{A}\mathcal{B}''}. 
\end{equation}
This completes the proof of Theorem \ref{dhall_real}.
\end{proof}

\subsection{Distance of the exceptional zero to 1-line}\label{real-zeros}
With Theorem \ref{dhall_real} in hand, we shall prove Theorem \ref{upperbd} in this section.  
\begin{proof}[Proof of Theorem \ref{upperbd}]
We let $\eta\in(0,1]$ be a parameter to be chosen later. 
We consider $\beta$ a non-exceptional real zero of $\zeta_L(s)$. 
We assume $\beta\ge 1-\eta$, and we consider $c_3$ such that
$$
1-\beta_1= d_L^{-c_3}.
$$
By Theorem \ref{dhall_real}, whenever $\Lo\ge\Lo_0$, we have
$$
\eta\ge 1-\beta \ge 
c_2^{\prime\prime}  \frac{ \log c_1^{\prime\prime} - \log (1-\beta_1)  -\log \log d_L }{\log d_L }
\ge  c_2^{\prime\prime}  \frac{ \log c_1^{\prime\prime} +c_3\Lo-\log \Lo}{\Lo},
$$
where $c_1^{\prime\prime}$ and $c_2^{\prime\prime}$ are defined in \eqref{def-c7'-c8'}.
Thus,
\begin{equation}\label{def-c10}
c_3
\le  \frac{\eta}{c_2^{\prime\prime}} +  \frac{ \log \Lo -\log c_1^{\prime\prime} }{\Lo}
 \le  \frac{\eta}{c_2^{\prime\prime}} +  \frac{ \log \Lo_0 -\log c_1^{\prime\prime} }{\Lo_0},
\end{equation}
which completes the proof of Theorem \ref{upperbd}.
\end{proof}

\begin{remark}
We note that if there are only two real zeros (i.e $\beta_1$ and $1-\beta_1$), then we can only take the trivial bound $\beta=1-\beta_1>0$ (i.e. $\eta=1$). However, if there is a third real zero $\beta'$, by the symmetry of zeros of $\zeta_L(s)$, we may take $\beta=\beta'\ge 1/2$ (so $\eta= 1/2$ is admissible), which reduces the admissible value for $c_3$ and, consequently, would push  $\beta_1$ away from  the 1-line a bit further. In other words, while Deuring-Heilbronn describes how an exceptional real zero would push further left other zeros in its vicinity, the other real zeros also do the same to the exceptional zero. We note that this phenomenon would hold for Dirichlet $L$-functions as well. 
\end{remark}

\subsection{Numerical results}

As Theorems \ref{thm-DH-all}, \ref{dhall_real}, and \ref{upperbd} are ``trivial'' if  $\zeta_L(s)$ does not admit the exceptional zero $\beta_1$, it is sufficient to calculate $c_1,c_2,c_1^{\prime},c_2^{\prime},c_1^{\prime\prime},c_2^{\prime\prime},c_3$ for $L$ such that $\zeta_L(s)$ may admit the exceptional zero $\beta_1$. To do so, we first summarize some cases that the non-existence of $\beta_1$  is known.

For $n_L=2$, we recall that there is a (quadratic) Dirichlet character $\chi_{L}$ modulo $d_L$ such that
$
\zeta_L(s)=\zeta(s)L(s,\chi_{L}),
$
where $\zeta(s)$ is the Riemann zeta function, and $L(s,\chi_{L})$ is the Dirichlet $L$-function attached to $\chi_{L}$. It is known that $\zeta(\sigma)$ is non-vanishing for $\sigma>0$. Moreover, Platt \cite{Pla16} showed that for any Dirichlet character $\chi$ modulo $q$, with $q\le 400\,000$,  $L(s,\chi)$, the Dirichlet $L$-function attached to $\chi$, has no positive real zeros. Hence, for $n_L=2$, $\zeta_L(s)$ has no positive real zeros if $d_L\le 400\,000$. 

For $n_L=3$ and $d_L\le 239$, Tollis \cite{T0l97} verified the generalized Riemann hypothesis for $\zeta_L(s)$ to $|\Im(s)|\le 92$. Also, for   $n_L=4$ and $d_L\le 320$, Tollis \cite{T0l97} verified the generalized Riemann hypothesis for $\zeta_L(s)$ to $|\Im(s)|\le 40$. Consequently, $\zeta_L(s)$  has no positive real zeros whenever $n_L=3$ and $d_L\le 239$ or $n_L=4$ and $d_L\le 320$. Hence, for $2\le n_L\le 4$, we only need to compute $c_1,c_2,c_1^{\prime},c_2^{\prime},c_1^{\prime\prime},c_2^{\prime\prime},c_3$ for $L$ with $d_L\ge d_0$, where $d_0$ given in  Table \ref{c1-c2-c3}.

For fields $L$ of degree $n_L\ge 5$, instead of trying to confirm  the non-existence of the exceptional zero of $\zeta_L(s)$, we use the following lower bounds $d_0$ of $d_L$ to calculate $c_1,c_2,c_1^{\prime},c_2^{\prime},c_1^{\prime\prime},c_2^{\prime\prime},c_3$. For $5\le n_L \le 8$, by \cite{LMFDB}, it can be checked that all the fields $L$ satisfy that $d_L\ge d_0$, where $d_0$ is given in  Table \ref{c1-c2-c3}.  For $9\le n_L \le 20$, as remarked by Fiori, by the lower bounds for the minimal root discriminant  in  \cite{DiazMinkowskiDiscriminantBounds}, we know that $d_L\ge d_0$, where $d_0$ is given in  Table \ref{c1-c2-c3}. In addition, for  $n_L \ge 21$, by \cite{DiazMinkowskiDiscriminantBounds}, one has
$
d_L> 10^{n_L},
$  
and thus $d_0$ can be taken as $d_0=10^{n_L}$ for these fields. 
\begin{center}
\begin{table}[h] 
\caption{Values for $(c_1,c_2,c_1^{\prime},c_2^{\prime},c_1^{\prime\prime},c_2^{\prime\prime},c_3)$ as defined in Theorems \ref{thm-DH-all}, \ref{dhall_real}, and \ref{upperbd}}
\label{c1-c2-c3}
\begin{tabular}{|c|c|c|c|c|c|c|c|c|} 
 \hline
 $n_0$ & $d_0 $ & $c_1$ & $c_2$ & $c_1^{\prime}$ & $c_2^{\prime}$ & $c_1^{\prime\prime}$ & $c_2^{\prime\prime}$ & $c_3$   \\ 
 \hline
 $ 2 $ & $ 400\,000 $ & $ 5.174 \cdot 10^{-5} $ & $ 0.04233 $ & $ 7.904 \cdot 10^{-5} $ & $ 0.06466 $ & $ 1.581 \cdot 10^{-4} $ & $ 0.1293 $ & $ 8.608 $  \\
$ 3 $ & $ 239$ & $ 1.276 \cdot 10^{-4}$ & $ 0.05697$ & $ 1.133 \cdot 10^{-4}$ & $ 0.05059$ & $ 2.275 \cdot 10^{-4}$ & $ 0.1015$ & $ 11.69 $ \\
$ 4$ & $ 320$ & $ 1.047 \cdot 10^{-4}$ & $ 0.04974$ & $ 1.052 \cdot 10^{-4}$ & $ 0.05000$ & $ 2.121 \cdot 10^{-4}$ & $ 0.1008$ & $ 11.69 $ \\
$ 5$ & $ 1\,609$ & $ 8.485 \cdot 10^{-5}$ & $ 0.04960$ & $ 8.790 \cdot 10^{-5}$ & $ 0.05139$ & $ 1.774 \cdot 10^{-4}$ & $ 0.1037$ & $ 11.08 $ \\
$ 6$ & $ 9\,747$ & $ 7.056 \cdot 10^{-5}$ & $ 0.05019$ & $ 7.371 \cdot 10^{-5}$ & $ 0.05243$ & $ 1.488 \cdot 10^{-4}$ & $ 0.1058$ & $ 10.65 $ \\
$ 7$ & $ 184\,607$ & $ 5.852 \cdot 10^{-5}$ & $ 0.05397$ & $ 5.802 \cdot 10^{-5}$ & $ 0.05351$ & $ 1.168 \cdot 10^{-4}$ & $ 0.1077$ & $ 10.23 $ \\
$ 8$ & $ 1\,257\,728$ & $ 5.118 \cdot 10^{-5}$ & $ 0.05411$ & $ 5.105 \cdot 10^{-5}$ & $ 0.05397$ & $ 1.028 \cdot 10^{-4}$ & $ 0.1087$ & $ 10.04 $ \\
$ 9$ & $ 2.290 \cdot 10^{7}$ & $ 4.450 \cdot 10^{-5}$ & $ 0.05620$ & $ 4.321 \cdot 10^{-5}$ & $ 0.05457$ & $ 8.690 \cdot 10^{-5}$ & $ 0.1097$ & $ 9.831 $ \\
$ 10$ & $ 1.560 \cdot 10^{8}$ & $ 4.008 \cdot 10^{-5}$ & $ 0.05609$ & $ 3.916 \cdot 10^{-5}$ & $ 0.05480$ & $ 7.878 \cdot 10^{-5}$ & $ 0.1102$ & $ 9.728 $ \\
$ 11$ & $ 3.910 \cdot 10^{9}$ & $ 3.602 \cdot 10^{-5}$ & $ 0.05792$ & $ 3.436 \cdot 10^{-5}$ & $ 0.05526$ & $ 6.906 \cdot 10^{-5}$ & $ 0.1110$ & $ 9.579 $ \\
$ 12$ & $ 2.740 \cdot 10^{10}$ & $ 3.323 \cdot 10^{-5}$ & $ 0.05774$ & $ 3.187 \cdot 10^{-5}$ & $ 0.05538$ & $ 6.406 \cdot 10^{-5}$ & $ 0.1113$ & $ 9.517 $ \\
$ 13$ & $ 7.560 \cdot 10^{11}$ & $ 3.037 \cdot 10^{-5}$ & $ 0.05914$ & $ 2.862 \cdot 10^{-5}$ & $ 0.05574$ & $ 5.749 \cdot 10^{-5}$ & $ 0.1120$ & $ 9.410 $ \\
$ 14$ & $ 5.430 \cdot 10^{12}$ & $ 2.755 \cdot 10^{-5}$ & $ 0.05899$ & $ 2.607 \cdot 10^{-5}$ & $ 0.05582$ & $ 5.236 \cdot 10^{-5}$ & $ 0.1121$ & $ 9.370 $ \\
$ 15$ & $ 1.610 \cdot 10^{14}$ & $ 2.527 \cdot 10^{-5}$ & $ 0.06010$ & $ 2.359 \cdot 10^{-5}$ & $ 0.05610$ & $ 4.736 \cdot 10^{-5}$ & $ 0.1126$ & $ 9.288 $ \\
$ 16$ & $ 1.170 \cdot 10^{15}$ & $ 2.424 \cdot 10^{-5}$ & $ 0.05987$ & $ 2.273 \cdot 10^{-5}$ & $ 0.05614$ & $ 4.565 \cdot 10^{-5}$ & $ 0.1127$ & $ 9.261 $ \\
$ 17$ & $ 3.700 \cdot 10^{16}$ & $ 2.273 \cdot 10^{-5}$ & $ 0.06080$ & $ 2.108 \cdot 10^{-5}$ & $ 0.05638$ & $ 4.231 \cdot 10^{-5}$ & $ 0.1132$ & $ 9.196 $ \\
$ 18$ & $ 2.730 \cdot 10^{17}$ & $ 2.172 \cdot 10^{-5}$ & $ 0.06062$ & $ 2.021 \cdot 10^{-5}$ & $ 0.05639$ & $ 4.057 \cdot 10^{-5}$ & $ 0.1132$ & $ 9.177 $ \\
$ 19$ & $ 9.030 \cdot 10^{18}$ & $ 2.010 \cdot 10^{-5}$ & $ 0.06140$ & $ 1.852 \cdot 10^{-5}$ & $ 0.05660$ & $ 3.718 \cdot 10^{-5}$ & $ 0.1136$ & $ 9.123 $ \\
$ 20$ & $ 6.740 \cdot 10^{19}$ & $ 1.908 \cdot 10^{-5}$ & $ 0.06122$ & $ 1.765 \cdot 10^{-5}$ & $ 0.05661$ & $ 3.542 \cdot 10^{-5}$ & $ 0.1136$ & $ 9.109 $ \\
$ 21+$ & $ 10^{n_L}$ & $ 1.819 \cdot 10^{-5}$ & $ 0.06141$ & $ 1.679 \cdot 10^{-5}$ & $ 0.05669$ & $ 3.370 \cdot 10^{-5}$ & $ 0.1138$ & $ 9.082 $ \\
 \hline
\end{tabular}
\end{table}
\end{center}

\section{The least prime in the Chebotarev density theorem - Proof of Theorem \ref{mainthm}}\label{leastprime}

\subsection{Choosing a weight to detect the least prime}\label{weight}

Let $x\ge 1$.   Let $\theta>1$ be a parameter to be chosen later. We consider the kernel
\begin{equation}\label{def-k}
k(s)=k_{\theta}(s)=\Big(\frac{x^{\theta( s-1)}-x^{s-1}}{s-1}\Big)^2
\end{equation}
and recall its inverse Mellin
transform is
\begin{equation}\label{def-hatk}
 \hat{k}(u) =\frac{1}{2\pi i}\int_{2-i\infty}^{2+i\infty}k(s) u^{-s}ds
 =\left.
  \begin{cases}
   u^{-1}\log( x^{2\theta}/{u}) & \text{if } x^{\theta+1} \le u \le x^{2\theta}; \\
   u^{-1}\log({u}/{x^2}) & \text{if } x^2 \le u \le x^{\theta+1};\\
   0 & \text{otherwise.} 
  \end{cases}
  \right.
\end{equation}
In \cite{AK,LMO}, $\theta$ was chosen to be 2. 
Note that 
\begin{align}
&  \hat{k}(u) \le \frac{(\theta -1)(\log x)}{u} ,  \label{bnd-khat}\\
& k(1)=k_{\theta}(1)=( (\theta-1) \log x )^2, \label{bnd-k1} \\
& k(\s)=\Big(\frac{x^{\theta( \s-1)}-x^{\s-1}}{\s-1}\Big)^2 
= \Big(\frac{1-x^{(\theta-1)( \s-1)}}{\s-1}\Big)^2 x^{2(\s-1)}
\le \frac{1}{(\s-1)^2}  x^{2(\s-1)} \ \text{if}\ \s<1
, \label{bnd-ksigma<1}\\
& |k(\s+it)| \le \frac{\Big(x^{\theta( \s -1)}+ x^{\s -1}\Big)^2}{|s-1|^2} 
\le k(\s ) \Big( \frac{1+x^{(\theta-1)( \s -1)}}{1- x^{(\theta-1)( \s -1)} }\Big)^2 \frac{(\s  -1)^2}{(\s-1)^2+t^2} 
, \label{bnd-kgal}\\
& \label{k-k-bd}
k(1)-k(\beta_1) 
= (\log x)^2 \phi_\theta( (1-\beta_1) \log x), \ \text{where}\  \phi_{\theta}(v) = (\theta-1)^2 - \Big(\frac{e^{-v}-e^{-\theta v}}{v} \Big)^2.
\end{align} 
We shall require the following lemma regarding the properties of $\phi_{\theta}$.

\begin{lemma}\label{phi-lemma}
Let $\theta>1$. Then 
\begin{enumerate}
 \item[(i)] $ \phi_{\theta}(v)$ is increasing for $v>0$;
 \item[(ii)] for any $v>0$, we have  $\phi_\theta(v) \ge (\theta-1)^2 (1- e^{-2 v})$. \\
 In particular, if $b>0$, then for any $0\le v\le b$, we have
 $
 \phi_\theta(v) \ge 2(\theta-1)^2 e^{-2b}v.
 $
\end{enumerate}
\end{lemma}

\begin{proof}[Proof of Lemma \ref{phi-lemma}]
We first note that
\begin{align*}
 \phi'_{\theta}(v)
 &= - 2 \Big(\frac{e^{-v}-e^{-\theta v}}{v} \Big)( (-v^{-2})(e^{-v}-e^{-\theta v}) + (v^{-1})     (-e^{-v} +\theta e^{-\theta v} ) )\\
&= 2 \Big(\frac{e^{-v}-e^{-\theta v}}{v^2} \Big) \Big(e^{-v}\Big(1+\frac{1}{v}\Big) -   e^{-\theta v}  \Big(\theta+\frac{1}{v}\Big) \Big).
\end{align*}
We claim $\phi'_{\theta}(v)\ge 0$ for $v\ge 0$ and so $\phi_{\theta}$ is increasing.
More precisely, letting $f(\theta)=  e^{-\theta v}  (\theta+\frac{1}{v} )$, we claim $f(1)\ge f(\theta)$ whenever $\theta> 1$. 
This follows from
$
f'(\theta)=  -v e^{-\theta v}   (\theta+\frac{1}{v} ) + e^{-\theta v}
 = -\theta v  e^{-\theta v} \le  0.
$
\\
Secondly, by the mean value theorem, for $\theta>1$, we have
$$
\Big|\frac{e^{-v}-e^{-\theta v}}{1-\theta} \Big|
 =  \max_{t\in[1,\theta]}| v e^{-t v}|\le   v e^{-v},
$$
and thus
$
\big(\frac{e^{-v}-e^{-\theta v}}{v}\big)^2 \le (\theta-1)^2 e^{-2v}.
$
Hence, we obtain
$$
 \phi_{\theta}(v) = (\theta-1)^2 - \Big(\frac{e^{-v}-e^{-\theta v}}{v} \Big)^2 
 \ge (\theta-1)^2 (1 - e^{-2v}).
$$
Lastly, for $c>0$, letting $g_c(v)= (1- e^{-c v}) - cve^{-cb} $, we note that $g(v)\ge 0$ for all $v\in[0,b]$. Indeed, as $g_c(0)=0$ and $g_c'(v)= c e^{-c v} -  ce^{-cb}\ge 0$ for all $v\le b$, we know that $g_c(v)$ is non-negative and increasing. Now, the last part of the lemma follows from using $g_c(v)$ with $ c=2$.
\end{proof}

\subsection{Explicit inequalities}\label{expl-ineq}

Throughout our discussion (from Section \ref{expl-ineq} to Section \ref{small-exp-zero}),  we let $L/K$ be a Galois extension of number fields with Galois group $G$, and we let  $C$ be a conjugacy class in $G$. We let $\mathcal{P}(C)$ denote the set of the unramified primes $\mathfrak{p} \subset \mathcal{O}_K$ of degree one such that $\sigma_{\mathfrak{p}}=C$. In addition, we choose $d_0$ according to Table \ref{table} and assume $\log d_L\ge \Lo_0= \log d_0$. We let $x\ge 101$ and let $k$ be a weight to be chosen later. We recall that $\beta_1$ denotes the possible exceptional zero, appearing in $(1-\frac{1}{R_0\Lo},1)$, of $\zeta_L(s)$. By \eqref{zfr-enlarged-R1}, we know that there is at most one zero of $\zeta_L(s)$ in $(1-\frac{1}{R_1\Lo},1- \frac{1}{R_0\Lo}]$; we shall denote such a zero by $\beta'$ if it exists.

Following \cite{AK} and \cite{LMO}, our goal is to show that there exists $c_4>0$ such that for $ x =  d_L^{c_4} = e^{c_4\Lo}$, we have
\begin{equation}\label{def-mathcalSC}
\mathcal{S}_C=  \sum_{\mathfrak{p}\in \mathcal{P}(C)} (\log N\mathfrak{p}) \hat{k} (N\mathfrak{p})>0,
\end{equation}
which implies that  there is a prime $\mathfrak{p}\in \mathcal{P}(C)$  such that 
\[
N\mathfrak{p}\le x^{2\theta} = d_L^{2\theta c_4}.
\]  
Thus we want to estimate
\begin{equation}\label{def-B}
B = 2\theta c_4.
\end{equation}

\begin{proposition}\label{prop-explicit-ineq}
Let  $k(s)$ be the kernel  given as in \eqref{def-k}, with $ x =  d_L^{c_4} \ge 101$, and $\mathcal{S}_C$ be the  sum  given as in \eqref{def-mathcalSC}. 
Then we have
\begin{align*}
\frac{|G|}{|C|}\mathcal{S}_C
\ge & 
(1-\delta(\beta_1)) (\theta-1)^2 (\log x )^2 
+  \delta(\beta_1) \phi_\theta( (1-\beta_1) \log x)  (\log x)^2  -(1-\delta(\beta_1))\delta(\beta')|k(\beta')|\\
&- \sum_{\rho \in \mathcal{S}\backslash \{\beta_1,\beta', 1-\beta_1\}}|k(\rho )| 
- \alpha_3 (\theta-1) \frac{|G|}{|C|}\Lo \frac{\log x}{x}
- \delta(\beta_1)  \frac{1}{( 1- \frac{1}{2\Lo} )^2} x^{-2(1- \frac{1}{2\Lo})}
-  c_5\Lo x^{-2} ,
\end{align*}
where $\phi_{\theta}, \alpha_3$, and $c_5$ are defined in \eqref{k-k-bd}, \eqref{alpha3}, and \eqref{c15}, respectively.
Here $\delta(\beta_1)=1$ if the exceptional zero $\beta_1$ exists for $\zeta_L(s)$, and $\delta(\beta_1)=0$ otherwise. $\delta(\beta')$ is defined the same way for $\beta_1$.
\end{proposition}

\begin{proof}
Let $g_{C}\in G$ be a representative of $C$, and let
 \begin{equation}\label{def-psiC}
 \Psi_{C}(s) = -\frac{|C|}{|G|} \sum_{\chi} \overline{\chi}(g_{C}) \frac{L'}{L} (s,\chi,L/K),
\end{equation}
where the sum is over the irreducible characters $\chi$ of $G=\Gal(L/K)$, and $L(s,\chi,L/K)$ is the Artin $L$-function attached to $\chi$.  It follows from the orthogonality property of irreducible characters of $G$ that 
 \begin{equation}\label{def-IC}
I_C =  \frac{1}{2\pi i}\int_{2-i\infty}^{2+i\infty}  \Psi_{C}(s)k(s) ds =\sum_{\mathfrak{p}}\sum_{m=1}^{\infty} J(\mathfrak{p}^m)  (\log N\mathfrak{p}) \hat{k} (N\mathfrak{p}^m),
\end{equation}
where  for ramified primes $\mathfrak{p}$, $|J(\mathfrak{p}^m)|\le 1$, for unramified primes $\mathfrak{p}$, $J(\mathfrak{p}^m)=1$ if $\sigma_{\mathfrak{p}}^m=C$, and $J(\mathfrak{p}^m)= 0$ otherwise.
By \eqref{bnd-khat}, as argued in \cite[Lemmata 3.1(1) and 3.3(1)]{AK} (see also \cite[Lemmata 3.1-3.3]{LMO}), for $x\ge 2$, one has
\begin{equation}
\sum_{\text{$\mathfrak{p}$ ramified}}\sum_{m\ge 1}  J(\mathfrak{p}^m)  (\log N\mathfrak{p}) \hat{k} (N\mathfrak{p}^m)\le 2(\theta-1) \frac{\log x}{x^2}\log d_L.
\end{equation}
Also, for $x\ge 101$, one has
\begin{equation}\label{def-alpha0}
\sum_{\mathfrak{p},m} J(\mathfrak{p}^m)  (\log N\mathfrak{p}) \hat{k} (N\mathfrak{p}^m)\le  16.08 \alpha_0 (\theta-1)\frac{  \log x}{x}  n_K ,\ \text{with}\ \alpha_0=1.25506 ,
\end{equation}
where  the sum is over $\mathfrak{p},m$ such that $N\mathfrak{p}^m \neq p$ for any (rational) prime $p$. Hence, by Minkowski's bound, for $x\ge 101$, we have as in \cite[Proposition 3.5]{AK}
\begin{equation}\label{IC-SC-bd}
|I_C - \mathcal{S}_C|\le  \frac{2(\theta-1) \log x}{x^2}\log d_L + 16.08 \alpha_0 (\theta-1) n_K\frac{\log x}{x} \le \alpha_3(\theta-1) \frac{ \log x}{x} \log d_L,
\end{equation}
where  
\begin{equation}\label{alpha3}
 \alpha_3 =\frac{2}{101} +\frac{32.16\alpha_0}{\log 3} =36.7595\ldots.
\end{equation}
The second step is to relate $ \Psi_{C}(s)$ to the zeros of the Dedekind zeta function $\zeta_L(s)$. To do so, we recall that by Deuring's reduction \cite{De}, denoting the fixed field of $g_{C}$  by $E$, one has
\[
 \Psi_{C}(s) = -\frac{|C|}{|G|} \sum_{\psi} \overline{\psi}(g_{C}) \frac{L'}{L} (s,\psi,L/E),
\]
where the sum is over the irreducible characters $\psi$ of $\Gal(L/E)$. As $L/E$ is abelian, it follows from Artin reciprocity that each Artin $L$-function $L(s,\psi,L/E)$ corresponds to a Hecke $L$-function.  This allows us to use the classical method of contour integration to find a lower bound of $I_C$ (in terms of $k(s)$ and the zeros of $\zeta_L(s)$). Indeed, one has
$$
I_C  =\frac{|C|}{|G|} \sum_{\psi}\overline{\psi}(g_{C}) \Big( \frac{1}{2\pi i} \int_{2-i\infty}^{2+i\infty} -\frac{L'}{L} (s,\psi,L/E) k(s)ds \Big).
$$
Proceeding exactly as in \cite[Sec. 4]{AK} and \cite[Sec. 3]{LMO} (see also \cite{LO}), an application of Cauchy's integral formula gives
\begin{equation}\label{line-int-bd}
\frac{|G|}{|C|}I_C \ge k(1) -\sum_{\rho \in \mathcal{S}}|k(\rho )| - n_L k(0) - \sum_{\psi} |V(\psi)|,
\end{equation}
where
$$
V(\psi) = \frac{1}{2\pi i} \int_{-\frac{1}{2}-i\infty}^{-\frac{1}{2}+i\infty} -\frac{L'}{L} (s,\psi,L/E) k(s)ds.
$$
Denoting $A(\psi)=d_E N_{E/\Bbb{Q}}(\mathfrak{f}_{\psi})$, where $\mathfrak{f}_{\psi}$ is the conductor of $\psi$, we appeal to a bound for the $L$-term of the form as given in \cite[Lemma 6.2]{LO} and as given explicitly in \cite[Lemma 2.19]{DasMSc}: 
\begin{equation}\label{def-v}
\Big|\frac{L'}{L} (-\frac{1}{2} +it,\psi,L/E) \Big| \le \log A(\psi) +n_E v(t),
\ \text{with}\ 
v(t)=\log (\sqrt{0.25 +t^2} +1)+ 4.452+ \frac{83}5,
\end{equation}
with $4.452+ \frac{83}5 =21.052$.
(Note this is an improvement on \cite[Lemma 5.1]{Wi} where instead $v(t)=\log (\sqrt{0.25 +t^2} +2)+\frac{19683}{812}$.)

Using the bound \eqref{bnd-kgal} for $k$, we have  for any $x\ge 101$,
\[
\Big|k\Big(-\frac{1}{2}+it\Big)\Big|
\le k\Big(-\frac{1}{2}\Big) \Big(\frac{1 +101^{-\frac{3}{2}(\theta-1)}}{  1 -101^{-\frac{3}{2}(\theta-1)}}\Big)^2   \frac{9}{9+ 4t^2}.
\]
Hence, we deduce
$$
|V(\psi)| \le \frac{1}{2 \pi} k\Big(-\frac{1}{2}\Big) \Big(\frac{1 +101^{-\frac{3}{2}(\theta-1)}}{  1 -101^{-\frac{3}{2}(\theta-1)}}\Big)^2 \int_{-\infty}^{\infty}  (\log A(\psi) +n_E v(t))     \frac{9}{9+ 4t^2} dt.
$$ 
Now, from \eqref{line-int-bd} and  the conductor-discriminant formula $\sum_\psi \log A(\psi) =\log d_L$, 
\begin{equation}\label{lower-bd-IC}
\frac{|G|}{|C|}I_C \ge k(1) -\sum_{\rho \in \mathcal{S}}|k(\rho )| -W_0 k\Big(-\frac{1}{2}\Big)\log d_L - n_L\Big(k(0)+W_1 k\Big(-\frac{1}{2}\Big) \Big),
\end{equation}
where, as before, $\mathcal{S}$ denotes the set of all non-trivial zeros of $\zeta_L(s)$, 
\begin{equation}\label{def-mu&nu}
W_0=\frac{1}{\pi }\int_0^{\infty} W(t)dt,\enspace W_1=\frac{1}{\pi }\int_0^{\infty} v(t)W(t)dt, \enspace W(t)= \Big(\frac{1 +101^{-\frac{3}{2}(\theta-1)}}{  1 -101^{-\frac{3}{2}(\theta-1)}}\Big)^2  \frac{9}{9+ 4t^2},
\end{equation}
and $v(t)$ is defined in \eqref{def-v}. 
Together with Minkowski's bound and \eqref{bnd-ksigma<1}, we deduce,
for $x\ge 101$, 
$$
W_0 k\Big(-\frac{1}{2}\Big)\log d_L +n_L \Big( k(0) +W_1 k\Big(-\frac{1}{2}\Big) \Big)
\le c_5 (\log d_L)x^{-2},
$$
where
\begin{equation}\label{c15}
c_5 =\frac{2}{\log 3} +\frac{4}{909} \Big( W_0 +\frac{2}{\log 3}W_1 \Big).
\end{equation}
Hence, putting \eqref{IC-SC-bd} and \eqref{lower-bd-IC} together, if $\zeta_L(s)$ has no exceptional zero, we have
\begin{equation}\label{lower-bd-k-no-beta1}
\frac{|G|}{|C|}\mathcal{S}_C\ge k(1) -\sum_{\rho \in \mathcal{S}}|k(\rho )| - \frac{c_5\Lo}{x^2} -\alpha_3 (\theta-1) \frac{|G|}{|C|} \frac{\log x}{x}\Lo,
\end{equation} 
where  $c_5$  and $\alpha_3$  are defined  in  \eqref{c15} and \eqref{alpha3}.
We then use  \eqref{bnd-k1} to bound $k(1)$. 
Otherwise, if $\zeta_L(s)$ admits an exceptional zero $\beta_1$, the term $k(1)-k(\beta_1)$ appears. By \eqref{bnd-ksigma<1} and \eqref{k-k-bd}, we know
\begin{equation}
\label{bd-k-beta1} 
k(1)- k(\beta_1) = (\log x)^2 \phi_\theta( (1-\beta_1) \log x) \ \text{and}\ 
|k(1-\beta_1)|\le \frac{x^{-2\beta_1}}{\beta_1^2}\le  \frac{x^{-2(1- \frac{1}{2\Lo})}}{( 1- \frac{1}{2\Lo} )^2}. 
\end{equation}
Thus, in the case the exceptional zero $\beta_1$ appears, we have
\begin{align}\label{lower-bd-k}
 \begin{split}
\frac{|G|}{|C|}\mathcal{S}_C
\ge &  (\log x)^2 \phi_\theta( (1-\beta_1) \log x) +  \frac{1}{( 1- \frac{1}{2\Lo} )^2} x^{-2(1- \frac{1}{2\Lo})}\\
&- \sum_{\rho \in \mathcal{S}\backslash \{\beta_1, 1-\beta_1\}}|k(\rho )|
  - \frac{c_5\Lo}{x^2} -\alpha_3 (\theta-1) \frac{|G|}{|C|} \frac{\log x}{x}\Lo.
 \end{split}
\end{align}
Finally, writing the sum in \eqref{lower-bd-k} as
$$
\sum_{\rho \in \mathcal{S}\backslash \{\beta_1, 1-\beta_1\}}|k(\rho )|=(1-\delta(\beta_1))\delta(\beta')|k(\beta')| + \sum_{\rho \in \mathcal{S}\backslash \{\beta_1,\beta' , 1-\beta_1\}}|k(\rho )|,
$$
we conclude the proof.
\end{proof}
We now focus on controlling the size of the above sums over the zeros. This is done by means of  zero-free regions, Deuring-Heilbronn repulsion, and zero-density estimates for zeros of $\zeta_L(s)$. 
First, we have some explicit estimate for $N_L (T)$ which counts the number of non-trivial zeros $\rho = \beta+i\gamma$ of  $\zeta_L(s)$ such that  $|\gamma| \leq T$.
We refer to the recent improvement of \cite{HSW}:
\begin{lemma}
Let $T\ge 1$. Then
\begin{equation}\label{zero-density}
\Big| N_L (T)  - \frac{T}{\pi} \log \Big( d_L \Big( \frac{T}{2\pi e}\Big)^{n_L}\Big)\Big|
 \le   0.296 (\log d_L + n_L \log T) +  3.971  n_L  +  3.969.
\end{equation}
\end{lemma}
In addition, we require a sharper estimate for the number of non-trivial zeros close to $1$.
Denoting $Z(r;s)=\{\rho \in\mathcal{S} \mid |\rho -s |\le r \}$ , we consider $n(r;s)=| Z(r;s)|$, the number of non-trivial zeros $\rho $ of $\zeta_L(s)$ such that $|\rho -s|\le r$. 

\begin{lemma}\label{counting-zero}
Let $n_0\ge2, \mathscr{Q}_0>0, 0<r\le 1$, and $\alpha>0$. If $n_0 \le n_L\le \mathscr{Q}_0 \Lo$, then 
\begin{equation}\label{counting-zero-bd-fixed}
n(r;1)
\le \Big( \frac{1+\alpha}{\alpha}\Big)^2( 1+  \alpha r  \omega(\alpha)  \Lo),
\end{equation}
where
\begin{equation}\label{def-omega}
\omega(\alpha) =\frac{1}{2} + \frac{\mathscr{Q}_0}{2} \max\Big\{ \frac{\Gamma'}{\Gamma} \Big( \frac{2+\alpha}{2} \Big) -\log\pi +\frac{2}{n_0} ,0 \Big\}  .
\end{equation}
\end{lemma}

\begin{proof}
We let $0<r\le 1$ and $\alpha>0$ and set $s_0=1+\alpha r$. 
Putting together the classical explicit formula (see, e.g., \cite[p. 1442]{AK})
\begin{equation}\label{exp-for}
\sum_{\rho \in \mathcal{S}} \Re \frac 1{s_0-\rho }
= \frac12 \log d_L + \Re\Big( \frac1s_0 + \frac1{s_0-1}\Big) + \Re \frac{\gamma'_L}{\gamma_L}(s_0) + \Re \frac{\zeta'_L}{\zeta_L}(s_0) ,
\end{equation}
with the facts that
$
\Re \frac{\zeta'_L}{\zeta_L}(s_0) =
  \frac{\zeta'_L}{\zeta_L}(1+\alpha r) \le 0,
$
and 
\[
\frac{\gamma'_L}{\gamma_L}(s_0)
= \frac{(r_1+r_2)}{2}\frac{\Gamma'}{\Gamma} \Big( \frac{s_0}{2} \Big)
+\frac{r_2}2 \frac{\Gamma'}{\Gamma} \Big( \frac{s_0+1}{2} \Big)
-\frac{n_L}2 \log \pi
\le \frac{n_L}{2}\Big( \frac{\Gamma'}{\Gamma} \Big( \frac{2+\alpha}{2} \Big) -\log\pi \Big),
\]
yields
\[
\sum_{\rho \in \mathcal{S}} \Re \frac 1{s_0-\rho }
\le  \frac{1}{\alpha r}  + \frac12 \log d_L +  \frac{n_L}{2}\Big( \frac{\Gamma'}{\Gamma} \Big( \frac{2+\alpha}{2} \Big) -\log\pi \Big)  + 1.
\]
Note that
$$
\sum_{\rho \in \mathcal{S}} \Re \frac 1{s_0-\rho }
\ge \sum_{\rho \in Z((1+\alpha)r;s_0)} \Re \frac 1{s_0-\rho }\ge \frac{\alpha r}{( (1+\alpha) r)^2} n((1+\alpha)r; s_0) =\frac{\alpha }{ (1+\alpha)^2 r} n((1+\alpha)r; s_0).
$$
Observing that $Z(r;1)\subseteq Z((1+\alpha)r;1+\alpha r)= Z((1+\alpha)r;s_0)$, \eqref{counting-zero-bd-fixed} follows from 
\[
n(r;1) 
\le n((1+\alpha)r;s_0)
\le \frac{ (1+\alpha)^2 r}{\alpha } \Big(  \frac{1}{\alpha r}  + \frac12 \log d_L +  \frac{n_L}{2} \Big( \frac{\Gamma'}{\Gamma} \Big( \frac{2+\alpha}{2} \Big) -\log\pi \Big)  + 1 \Big) . \qedhere
\]
\end{proof}

In addition, we require the following result on sums over the zeros close to $1$.

\begin{corollary}\label{cor-zero-sums}
Assume \eqref{assumptions}, $\alpha>0, r>0$ and that $\frac{1}{r\Lo} \le 1$. If $n_0 \le n_L\le \mathscr{Q}_0 \Lo$, then 

\begin{equation}
\sum_{\frac{1}{r\Lo }\le|\rho  -1|\le 1}  \frac1{|\rho -1|^2}
 \le 
\Big( \Big( \frac{1+\alpha}{\alpha}\Big)^2  ( r^2 + 2r\alpha\omega(\alpha) )  - n(\frac1{r\Lo};1) r^2 \Big)\Lo^2,
\end{equation}
where $\omega(\alpha)$, depending on $n_0$ and $\mathscr{Q}_0$, is defined as in \eqref{def-omega}.
\end{corollary}

\begin{proof}
We start with 
\begin{equation}\label{part1}
\sum_{\frac{1}{r\Lo }\le|\rho  -1|\le 1}  \frac1{|\rho -1|^2}
=
\int_{ \frac{1}{r\Lo} }^1\frac{1}{u^2} d n(u;1) 
= n(1;1) - n\Big(\frac1{r\Lo};1\Big) (r\Lo)^2 + 2\int_{ \frac{1}{r\Lo} }^1\frac{n(u;1) }{u^3} du . 
\end{equation}
It follows from Lemma \ref{counting-zero} that
\begin{equation}\label{part2}
n(1;1)\le \Big(\frac{1+\alpha}{\alpha}\Big)^2(\alpha\omega(\alpha)\Lo +1  ),
\text{and} 
\int_{\frac{1}{r\Lo}}^1  \frac{   1+  \alpha u \omega(\alpha)  \Lo }{u^3} du   
=  \Big(\frac{r^2}{2}+\alpha\omega(\alpha)r\Big)\Lo^2 - \alpha\omega(\alpha) \Lo-\frac{1}{2}.
\end{equation}
We conclude by putting together \eqref{part1} and \eqref{part2}. 
\end{proof}

We now control non-trivial zeros $\rho =\beta+ i\gamma$ of $\zeta_L(s)$ that are not equal to $\beta_1$, $\beta'$, or $1-\beta_1$.

\begin{proposition}\label{prop-explicit-ineq-zeros}
Let $k(s)$ be the kernel  given as in \eqref{def-k}, with $ x =  d_L^{c_4}$, and $\mathcal{S}_C$ be the sum  given as in \eqref{def-mathcalSC}. Let $R_1$ be defined as in \eqref{zfr-enlarged-R1}. 
If $n_0\le n_L \le \mathscr{Q}_0\Lo$, then
\begin{align*}
\frac{|G|}{|C|}\mathcal{S}_C\ge &  
(1-\delta(\beta_1)) (\theta-1)^2 c_4^2 \Lo^2
+  \delta(\beta_1) \phi_\theta\Big(  c_4(1-\beta_1) \Lo\Big) c_4^2 \Lo^2\\
&- (1-\delta(\beta_1))\delta(\beta')|k(\beta')|
-c_6 \Lo 
- \sideset{}{'}\sum_{\frac{1}{R_1\Lo }\le|\rho  -1|\le 1} |k(\rho )| \\  
&- \cQ_0 \alpha_3 (\theta-1) c_4 \Lo^3 e^{-c_4 \Lo }
- \delta(\beta_1)  \frac{1}{( 1- \frac{1}{2\Lo} )^2} e^{ - c_4(2 \Lo - 1)}
-  c_5\Lo e^{-2 c_4 \Lo },
\end{align*}
where $\delta(\beta_1)=1$ (resp., $\delta(\beta')=1$) if the exceptional zero $\beta_1$ (resp., real zero $\beta'$) exists for $\zeta_L(s)$, $\delta(\beta_1)=0$ (resp., $\delta(\beta')=0$) otherwise, $\phi_{\theta}$, $\alpha_3$, $c_5$, and $c_6$ are defined in \eqref{k-k-bd}, \eqref{alpha3}, \eqref{c15}, and \eqref{c13}, respectively, and, as later, the primed sum is over non-trivial zeros $\rho\neq 1-\beta_1$ of $\zeta_L(s)$.  
\end{proposition}

\begin{proof}
We shall note that if $\rho =\beta +i\gamma$ is such that $|\rho  -1|< \frac{1}{R_1\Lo}$, then  $1-\beta < \frac{1}{R_1\Lo}$ and $|\gamma|<\frac{1}{R_1\Lo}$. 
Hence, \eqref{zfr-enlarged-R1} forces that  $\rho$ is either  $\beta_1$ or $\beta'$ if $|\rho  -1|< \frac{1}{R_1\Lo}$. 
Thus, we can consider the splitting
\begin{equation}\label{sum-non-beta1}
\sum_{\rho \in \mathcal{S} \backslash \{ \beta_1,\beta', 1-\beta_1\}} |k(\rho )| 
= \sideset{}{'}\sum_{\frac{1}{R_1\Lo }\le|\rho  -1|\le 1} |k(\rho )|  +\sum_{|\rho  -1|> 1} |k(\rho )|.
\end{equation}
To bound $\sum_{|\rho  -1|> 1} |k(\rho )|$, we appeal to \eqref{bnd-kgal} to bound $k$: 
\begin{equation}\label{bnd-krho>1}
|k(\rho )|\le \frac{(x^{\theta(\beta-1)}+ x^{\beta-1})^2}{|\rho -1|^2} 
\le \frac{4  }{|\rho -1|^2}.
\end{equation}
Recalling that $n(r;1)$ denotes the number of non-trivial zeros $\rho$ of  $\zeta_L(s)$ such that  $|\rho -1| \leq r$, we notice that $n(r;1) \le  N_L (r)$. 
Using \eqref{zero-density}, we derive
\[
\sum_{|\rho -1|>1} |k(\rho )| 
\le 4 \int_{1}^{\infty} \frac{1}{r^2}d n(r;1) 
\le  8  \int_{1}^{\infty} \frac{n(r;1)}{r^3}  dr
\le c_6 \Lo,
\]
where 
\begin{equation}\label{c13}
c_6=8\Big(\frac{1}{\pi}+ 0.148 + \cQ_0 \max\Big\{ 0, \int_{1}^{\infty} \frac{\frac{r}{\pi} \log( \frac{r}{2\pi e})  +0.296 \log r +3.971 + \frac{3.969}{n_0}}{r^3}dr \Big\}\Big).
\end{equation}
We conclude by putting together Proposition \ref{prop-explicit-ineq} with the above and the facts that $\frac{|G|}{|C|} \le n_L $ and $x=e^{c_4\Lo}$.
\end{proof}
Now, it remains to bound $ \sum'_{\frac{1}{R_1\Lo }\le|\rho  -1|\le 1} |k(\rho )|$. 
We shall proceed with the argument by considering cases depending on the existence of the exceptional zero $\beta_1$ and the distance of the exceptional zero $\beta_1$, if it exists, to the $1$-line.

\subsection{Non-exceptional case}

\begin{proposition}
Let $\mathcal{S}_C$ be the  sum  given as in \eqref{def-mathcalSC}. Assume 
there is no exceptional zero $\beta_1$ in $(1-\frac{1}{R_0\Lo},1)$. If $n_0\le n_L \le \mathscr{Q}_0\Lo$ and $\Lo\ge \Lo_0$, then
\begin{equation}\label{solving_c4}
\Lo^{-2}\frac{|G|}{|C|}  \mathcal{S}_C \ge (\theta-1)^2 c_4^2 - R_0 ^2 e^{-\frac{2c_4}{R_0}} 
-c_7(c_4) e^{\frac{-2 c_4}{29.57 (1+\Delta_0(1))}} - \mathcal{E}_0(c_4), 
\end{equation}
where 
\begin{equation}\label{def-epsilon0}
\mathcal{E}_0(t)= \frac{c_6}{\Lo_0} +  \mathscr{Q}_0 \alpha_3 (\theta-1) t \Lo_0 e^{-t \Lo_0 } 
+ \frac{c_5}{\Lo_0} e^{-2 t \Lo_0 } ,
\end{equation}
and  $c_5$,  $c_6$, and $c_7$ are defined as in  \eqref{c15}, \eqref{c13}, and \eqref{c14*}, respectively. 
\end{proposition}

\begin{proof}
We begin by recalling that 
$\beta' \le 1- \frac1{R_0\Lo}$ (if exists) and that for any non-trivial zero $\rho =\beta+i\gamma\neq \beta'$, with $|\gamma|\le 1$,  the zero-free region \cite[Proposition 6.1]{AK} gives
$$
1-\beta >(29.57 \log (d_L \tau^{n_L} ) )^{-1}\ge  (29.57\Lo  (1+\delta_L(3))  )^{-1} \ge  (29.57\Lo  (1+\Delta_0(1))  )^{-1},
$$
where $\delta_L$ and $\Delta_0$ are defined in \eqref{def-delta} and \eqref{def-Delta}, respectively, and the last inequality is due to \eqref{bnd-delta}. 
Therefore, by \eqref{bnd-ksigma<1}, we have
$$
|k(\beta')|\le   \frac{1}{(\beta'-1)^2}  x^{2(\beta'-1)} \le (R_0 \Lo)^2 x^{-\frac{2}{R_0\Lo }},
$$
and by using the first bound of \eqref{bnd-kgal}, we have for the other zeros that 
$$
|k(\rho )|\le \frac{x^{2(\beta -1)}  (1 + x^{(\theta-1)(\beta-1) } )^2  }{|\rho -1|^2}
 \le \frac{x^{-2  (29.57\Lo  (1+\Delta_0(1))  )^{-1}} (1 + x^{ (1-\theta) (29.57\Lo  (1+\Delta_0(1))  )^{-1} } )^2 }{|\rho -1|^2}.
$$
By Corollary \ref{cor-zero-sums}, putting $x=e^{c_4\Lo}$, we then derive  
\begin{equation}\label{sum<1'}
 (1-\delta(\beta_1))\delta(\beta')|k(\beta')| +\sideset{}{'}\sum_{\frac1{2\Lo}\le |\rho  -1|\le 1} |k(\rho )|
\le  (R_0\Lo )^2 e^{-\frac{2c_4}{R_0}} + c_7(c_4) \Lo^2 e^{\frac{-2 c_4}{29.57 (1+\Delta_0(1))}},
\end{equation}
where 
\begin{equation}\label{c14*}
c_7(t)
=  4 (1+\alpha\omega(\alpha)) \Big( \frac{1+\alpha}{\alpha}\Big)^2 \Big(1+ e^{ \frac{(1-\theta)t}{29.57 (1+\Delta_0(1))  } }  \Big)^2    ,
\end{equation}
and $\omega(\alpha)$ is defined as in \eqref{def-omega}.  For $ n_L\le \mathscr{Q}_0 \Lo$ , we apply Proposition \ref{prop-explicit-ineq-zeros} with $R_1=2$ and \eqref{sum<1'} to obtain
\begin{align*}
\frac{|G|}{|C|}\mathcal{S}_C
\ge &  
 (\theta-1)^2 c_4^2 \Lo^2
-c_6 \Lo 
- (R_0 \Lo)^2 e^{-\frac{2c_4}{R_0}} - c_7(c_4) \Lo^2 e^{\frac{-2 c_4}{29.57 (1+\Delta_0(1))}} \\ 
&-  \mathscr{Q}_0 \alpha_3 (\theta-1) c_4 \Lo^3 e^{-c_4 \Lo }
-  c_5\Lo e^{-2 c_4 \Lo }.
\end{align*}
We conclude with the fact that $\Lo\ge \Lo_0$.
\end{proof}

\subsection{Exceptional case}\label{small-exp-zero}

Assume that the exceptional zero $\beta_1$ of $\zeta_L(s)$ presents such that $\beta_1\ge 1-\frac{1}{R_0\Lo}$. 
Following Proposition \ref{lemma-zfr-enlarged}, we take  $R_0=20$ and $R_1=1.24$. We recall $R=29.57$.  
We let $\varepsilon_1>0$, $\sigma_1\ge 1$, and $\eta\in (0,1]$ be parameters (to be chosen later) to compute $c_1^{\prime}(\varepsilon_1,\sigma_1,\eta)$ and $c_2^{\prime}(\varepsilon_1,\sigma_1,\eta)$ defined in \eqref{def-C7-C8}. These parameters will be chosen to make $c_4$ and thus $B$ as small as possible.

\begin{proposition}\label{prop-exceptional}
Let $\mathcal{S}_C$ be the  sum  given as in \eqref{def-mathcalSC}. Assume that the exceptional zero $\beta_1$ of $\zeta_L(s)$ exists, that is $\beta_1\ge 1-\frac{1}{R_0\Lo}$. Let $\eta\in(0,1]$ satisfying
\begin{equation}\label{cond-eta}
\eta \ge  c_2^{\prime}(\varepsilon_1,\sigma_1,\eta) \frac{ \log \Big( \frac{c_1^{\prime}(\varepsilon_1,\sigma_1,\eta) }{(1-\beta_1) \Lo} \Big)}{\Lo} 
\end{equation}
(with $c_1^{\prime}(\varepsilon_1,\sigma_1,\eta)$ and $c_2^{\prime}(\varepsilon_1,\sigma_1,\eta)$ defined in \eqref{def-C7-C8}).
If $ n_0 \le n_L \le \mathscr{Q}_0 \Lo$ and $\Lo\ge \Lo_0$,
then
\begin{equation}\label{L8.3}
\begin{split} 
 \Lo^{-2}\frac{|G|}{|C|}\mathcal{S}_C
\ge &  
\phi_\theta( (1-\beta_1) c_4 \Lo) c_4^2 
- c_{11} ((1-\beta_1) \Lo  )^{2 c_4 c_8 }
- \frac{c_6}{ \Lo} \\
&- \cQ_0 \alpha_3 (\theta-1) c_4 \Lo e^{-c_4 \Lo }
- \frac{1}{(\Lo - \frac{1}{2} )^2} e^{ - c_4(2 \Lo - 1)}
- \frac{ c_5}{\Lo } e^{-2 c_4 \Lo},
\end{split} 
\end{equation}
where  $\alpha_3$, $c_5$, $c_6$, $c_8$, and  $c_{11}$ are defined in \eqref{alpha3}, \eqref{c15}, \eqref{c13}, \eqref{c12}, and \eqref{c14}, respectively.
\end{proposition}

\begin{proof} 
We start by estimating the sum over low-lying zeros $\sum'_{\frac{1}{R_1\Lo}\le |\rho  -1|\le 1} |k(\rho )|$. Let $\rho = \beta+ it$ denote a non-exceptional zero of $\zeta_L(s)$ such that $\frac{1}{R_1\Lo}\le |\rho  -1|\le 1$. We shall  consider the following two situations:
\begin{enumerate}
\item[(a)] We first  suppose that  $(1-\beta_1)\Lo \le c_1^{\prime}(\varepsilon_1,\sigma_1,\eta)^a$. Applying Theorem \ref{thm-DH-all}, we have either  $\beta\le 1-\eta$ or 
$$
1-\beta    \ge  c_2^{\prime}(\varepsilon_1,\sigma_1,\eta) \frac{ \log \Big( \frac{c_1^{\prime}(\varepsilon_1,\sigma_1,\eta)}{(1-\beta_1) \Lo } \Big)}{\Lo }  \ge  c_9 \frac{ \log\Big( \frac1{(1-\beta_1)\Lo  }\Big)}{\Lo },
$$
where
$$
c_9  = c_2^{\prime}(\varepsilon_1,\sigma_1,\eta) \Big( 1 -\frac{1}{a} \Big).
$$
\item[(b)] Secondly, we consider the situation that $(1- \beta_1)\Lo  > c_1^{\prime}(\varepsilon_1,\sigma_1,\eta)^a  $. By \cite[Proposition 6.1]{AK} and \eqref{bnd-delta}, we know that 
$$
1-\beta >(29.57 \log (d_L \tau^{n_L} ) )^{-1}\ge  (29.57\Lo  (1+\Delta_0(1))  )^{-1}
$$
as $\tau =|t|+2$ and $|t|\le 1$.
Thus, by choosing
$$
c_{10} = \frac{1}{ a \cdot 29.57(1+\Delta_0(1)) \log ( 1/ c_1^{\prime}(\varepsilon_1,\sigma_1,\eta))}, 
$$
we have
$$
1-\beta \ge c_{10} \frac{1}{  \Lo } \log   \Big( \frac{1}{c_1^{\prime}(\varepsilon_1,\sigma_1,\eta)^a }  \Big) \ge  c_{10} \frac{ \log\Big( \frac1{(1-\beta_1)\Lo  }\Big)}{\Lo }.
$$
\end{enumerate}

To summarise, for any non-exceptional zero $\rho =\beta+it$ such that  $\frac{1}{R_1\Lo}\le |\rho  -1|\le 1$, assuming \eqref{cond-eta}, we have
\begin{equation}
1-\beta \ge  D:= c_8 \frac{ \log\Big( \frac1{(1-\beta_1)\Lo } \Big)}{  \Lo },\ 
\text{with}\ 
c_8= \min \{c_9,c_{10} \}.
\end{equation}
We chose $a$ such that $  c_9 = c_{10}$, i.e.
\begin{equation}\label{def-a}
a = 1+\frac{1}{  29.57(1+\Delta_0(1))c_2^{\prime}(\varepsilon_1,\sigma_1,\eta)  \log ( 1/ c_1^{\prime}(\varepsilon_1,\sigma_1,\eta) )} .
\end{equation}
Thus, $c_8$ is given by
\begin{equation}\label{c12}
c_8 =
\frac{1}{ 1/c_2^{\prime}(\varepsilon_1,\sigma_1,\eta) + 29.57(1+\Delta_0(1)) \log ( 1/ c_1^{\prime}(\varepsilon_1,\sigma_1,\eta))}.
\end{equation}
Applying the bounds \eqref{bnd-kgal} on $k$ with $1-\beta\ge D$ gives
\begin{equation}\label{k-bd}
|k(\rho )|\le \frac{x^{2(\beta -1)}  (1 + x^{(\theta-1)(\beta-1) } )^2  }{|\rho -1|^2} 
\le \frac{x^{-2D} (1 + x^{ -(\theta-1)D } )^2 }{|\rho -1|^2} .
\end{equation}
Recalling that we set $x=d_L^{c_4}$, and by the definition of the exceptional zero with respect to the zero-free region \eqref{zfr-R0}, that $D  \ge c_8\frac{\log R_0}{\Lo } $, 
we have
\[
 x^{-2D} = ((1-\beta_1) \Lo  )^{2c_4 c_8 }
\ \text{and}\ 
 (1 + x^{ -(\theta-1)D } )^2 \le  \Big(1+ e^{-(\theta-1)c_4c_8 \log R_0 } \Big)^2.
\]
We apply Corollary \ref{cor-zero-sums} with $n(\frac1{R_1 \Lo};1)=1$:
\[
\sideset{}{'}\sum_{\frac{1}{R_1\Lo }\le|\rho  -1|\le 1}  \frac1{|\rho -1|^2}
\le
\Big(  \Big( \frac{1+\alpha}{\alpha}\Big)^2(R_1^2 + 2R_1\alpha\omega(\alpha)) -R_1^2 \Big)\Lo^2      ,
\]
which together with \eqref{k-bd}, gives
\begin{equation}\label{sum<1}
\sideset{}{'}\sum_{\frac{1}{R_1\Lo}\le |\rho  -1|\le 1} |k(\rho )|
\le
c_{11} \Lo^2 ((1-\beta_1) \Lo  )^{2c_4 c_8 },
 \end{equation}
with
\begin{equation}\label{c14}
c_{11}
=   \Big(1+ e^{-(\theta-1)c_4c_8 \log R_0 } \Big)^2  R_1\Big(  \Big( \frac{1+\alpha}{\alpha}\Big)^2(R_1+2\alpha\omega(\alpha))-R_1    \Big),
\end{equation}
where $\omega(\alpha)$, depending on $\mathscr{Q}_0$ and $n_0$, is defined in \eqref{def-omega}.
Note that $c_9, c_{10} $, and $c_{11}$ all depend on $\varepsilon_1,\sigma_1$, and $\eta$.
We conclude by using Proposition \ref{prop-explicit-ineq-zeros} and \eqref{sum<1}. 
\end{proof}
For the rest of the article, $n_0\le n_L \le \mathscr{Q}_0\Lo$ and $\Lo\ge \Lo_0$, with $\Lo_0$ and $\mathscr{Q}_0$ as defined in \eqref{def-L} and \eqref{def_Q}, respectively. 
For each $(n_0,d_0)$ listed in Table \ref{table}, we choose 
$\varepsilon_1>0 ,  \sigma_1\ge 1$, and $\eta\in (0,1]$ which give $c_1^{\prime}(\varepsilon_1,\sigma_1,\eta)$ and $ c_2^{\prime}(\varepsilon_1,\sigma_1,\eta)$ as defined in \eqref{def-C7-C8}. 
We then obtain $c_8$ as defined in \eqref{c12}, and define 
\begin{equation}\label{def-c4}
c_4 = \frac1{2c_8} + 0.001 .
\end{equation}
To prove Theorem \ref{mainthm}, we shall split our consideration depending on the distance of $\beta_1$ to 1-line.  
We shall introduce further parameters:
\[
\varepsilon_2>0, \ 
\sigma_2\ge1, \ 
\kappa \ge1, \ 
0< \lambda \le 1, \ 
0< \mu \le 1, \ \text{and}\ 
1< \nu \le 2.
\]
For the rest of the article, we denote 
\begin{equation}\label{def-C1C2}
C_1 = c_1^{\prime}(\varepsilon_2,\sigma_2,0.5)
\ \text{and}\ 
C_2 = c_2^{\prime}(\varepsilon_2,\sigma_2,0.5),
\end{equation}
where $c_1'$ and $c_2'$ are defined in \eqref{def-C7-C8}.
We assume 
\[
\frac{  (\kappa C_1  )^2}{ \Lo}
<
\frac{\mu}{c_4\Lo^{\nu-1} } 
<
\frac{\lambda}{c_4 } 
<
\frac{1}{R_0} .
\]
We explore the cases:
\begin{itemize}
\item ``medium'' when $\frac{\lambda}{c_4} \le (1-\beta_1)\Lo <  \frac{1}{R_0 } $,
\item ``small'' when $\frac{\mu}{ c_4 \Lo^{\nu-1}} \le (1-\beta_1 ) \Lo  \le \frac{\lambda}{c_4}$, 
\item ``very small'' when $  \frac{  (\kappa C_1  )^2}{ \Lo}  \le (1-\beta_1 ) \Lo  \le \frac{\mu}{c_4 \Lo^{\nu-1} }$, and 
\item ``extremely small'' when $(1-\beta_1)\Lo < \frac{  (\kappa C_1  )^2}{ \Lo} $.
\end{itemize}
For the ``extremely small'' case, we will use a different weight $k$ and thus leave this for the end. 
\\
In each of the first three cases,  we use the weight $k_{\theta}$ as defined in \eqref{def-k}, and  
obtain $B = 2\theta c_4$ as given by \eqref{def-B}.
The choice of parameters will ensure $B$ to be as small as possible while keeping the expression
in the right of \eqref{L8.3} positive. 
We note that if 
$(1-\beta_1)\Lo \ge \frac{  (\kappa C_1  )^2}{ \Lo} $, then 
the condition \eqref{cond-eta}, which is needed to apply Proposition \ref{prop-exceptional}, is satisfied when
\begin{equation}\label{cond-eta-2}
\frac{  c_2^{\prime}(\varepsilon_1,\sigma_1,\eta) }{\Lo_0} 
\Big( \max\Big\{  \log \Big( \frac{ c_1^{\prime}(\varepsilon_1,\sigma_1,\eta)}{ (\kappa C_1  )^2 } \Big),0     \Big\}  + \log  \Lo_0 \Big) \le \eta.
\end{equation}

\subsubsection{``Medium" case:}

Assume $\frac{\lambda}{c_4} \le (1-\beta_1)\Lo <  \frac{1}{R_0 } $. Lemma \ref{phi-lemma}(i) 
yields
\begin{equation}\label{phi-v>1'}
\phi_\theta( (1-\beta_1) c_4 \Lo) c_4^2 - c_{11} ((1-\beta_1) \Lo  )^{2c_4 c_8 }
\ge \phi_{\theta}(\lambda) c_4^2 - c_{11} \big((1-\beta_1) \Lo  \big)^{2c_4 c_8 }.
\end{equation}
Hence, Proposition \ref{prop-exceptional} gives 
\begin{equation}\label{solve-medium}
 \Lo^{-2}\frac{|G|}{|C|}\mathcal{S}_C
\ge  
 \phi_{\theta}(\lambda) c_4^2 
- c_{11} R_0^{-2c_4 c_8 }
 - \mathcal{E}_1(c_4),
\end{equation}
where $c_{11}$ is defined in \eqref{c14}, and $\mathcal{E}_1(t)$ is defined by
\begin{equation}\label{vareps1}
 \mathcal{E}_1(t) =  \frac{c_6}{ \Lo_0}+ \cQ_0 \alpha_3 (\theta-1) \Lo_0 t e^{-t \Lo_0 }
+ \frac{1}{(\Lo_0- \frac{1}{2} )^2} e^{ - t(2 \Lo_0 - 1)}
+ \frac{ c_5}{\Lo_0} e^{-2 t \Lo_0 }.
\end{equation}

\subsubsection{``Small" case:}

Assume $\frac{\mu}{c_4 \Lo^{\nu-1} } \le (1-\beta_1 ) \Lo  \le \frac{\lambda}{c_4}$.
As $(1-\beta_1 ) \Lo  \le \frac{\lambda}{c_4}$ and $2 c_4c_8 >1$, Lemma \ref{phi-lemma}(ii) yields
\[
\phi_\theta( (1-\beta_1) c_4 \Lo) c_4^2 - c_{11} ((1-\beta_1) \Lo  )^{2c_4 c_8 }
\ge 2(\theta-1)^2 e^{-2\lambda}  c_4^3(1-\beta_1 ) \Lo
 - c_{11} \Big(\frac{\lambda}{c_4}\Big)^{2 c_4c_8-1} ((1-\beta_1) \Lo  ).
\]
Hence, by Proposition \ref{prop-exceptional}, we have
\begin{equation}\label{solve-small}
((1-\beta_1 ) \Lo )^{-1} \Lo^{-2}\frac{|G|}{|C|}\mathcal{S}_C
\ge  
2(\theta-1)^2 e^{-2\lambda}   c_4^3
-  c_{11} \Big(\frac{\lambda}{c_4}\Big)^{2 c_4c_8-1} 
- c_6\frac{ c_4 \Lo_0^{\nu-2} }{\mu}
- \mathcal{E}_2(c_4) ,
\end{equation}
with
\begin{equation}\label{def-epsilon2}
\mathcal{E}_2(t)=   \Big(   
\cQ_0 \alpha_3 (\theta-1) \Lo_0^2 c_4 e^{- t \Lo_0}
+ \frac{1}{ \Lo_0 (1 - \frac{1}{2\Lo_0} )^2} e^{ - t (2 \Lo_0 - 1)}
+ c_5 e^{-2 t \Lo_0 } \Big) \frac{t\Lo_0^{\nu-2} }{\mu} .
\end{equation}

\subsubsection{``Very small" case:}

Assume $  \frac{ (\kappa C_1 )^2}{ \Lo}  \le (1-\beta_1 ) \Lo  \le \frac{\mu}{c_4\Lo^{\nu-1} }$. Then 
Lemma \ref{phi-lemma}(ii) yields 
\begin{multline*}
\phi_\theta( (1-\beta_1) c_4 \Lo) c_4^2 - c_{11} ((1-\beta_1) \Lo  )^{2c_4 c_8 }
\ge
 2(\theta-1)^2 e^{-\frac{2\mu}{\Lo_0^{\nu-1}}}   c_4^3 (1-\beta_1) \Lo
\\ - c_{11} \Big(\frac{\mu}{c_4  \Lo_0^{ \nu-1} } \Big)^{2c_4 c_8-1 } (1-\beta_1) \Lo,
\end{multline*}
which, combined with Proposition \ref{prop-exceptional}, yields 
\begin{equation}\label{solve-verysmall}
((1 -\beta_1) \Lo)^{-1}\Lo^{-2}\frac{|G|}{|C|}\mathcal{S}_C 
\ge  2(\theta-1)^2 e^{-\frac{2 \mu}{\Lo_0^{\nu-1}}}  c_4^3 -c_{11} \Big(\frac{\mu}{c_4 \Lo_0^{\nu-1} }\Big)^{2 c_4c_8-1}  
-   \frac{ c_6 }{ (\kappa C_1 )^2}
 -\mathcal{E}_3(c_4),
\end{equation}
with
\begin{equation}\label{def-epsilon3}
\mathcal{E}_3(t)=
\Big(  
\cQ_0 \alpha_3 (\theta-1) \Lo_0^2 t e^{-t \Lo_0 }
+ \frac{1}{\Lo_0 (1 - \frac{1}{2\Lo_0} )^2} e^{ - t (2 \Lo_0 - 1)}
+ c_5 e^{-2 t \Lo_0} \Big)  \frac{ 1}{ (\kappa C_1 )^2}.
\end{equation}

\subsubsection{``Extremely small" case:} 

We assume $1- \beta_1 < (\kappa C_1 )^2 \Lo^{-2}$.
Here we use the weight as introduced by \cite{LMO}, and used by \cite{AK}: 
\[
k(s)=x^{s^2+s},
\]
with $x=d_L^{c_{12}}\ge 10^{10}$. 
Thus, here $B$ is given by
\[
B = 5 c_{12}.\footnote{We observe that the power $5$ seems to be the smallest real number in order to control the sum over the higher power norm primes. Namely in \cite[Lemma 3.4]{AK2} we can see that the bound for the sum over the primes satisfying $N\mathfrak{p} > x^a$ is given by the integral $\int_{(a-1)(\log x)}^{\infty} \big(\frac{(t+\log x)t}{2\log x}-1\big) \frac{\exp\big(t-\frac{t^2}{4\log x}\big)}{\sqrt{4\pi \log x}} dt$. So for the very least, one needs a non-positive power in the exponent term, which means $1-\frac{(a-1)}{4}\le0$, giving $a\ge5$. }
\]
First, we have the following explicit inequality from combining 
\cite[(3.2), Proposition 4.7, line 9 of p. 1453, and line -1 of p. 1452]{AK}:
\begin{equation}\label{L8.5} 
 \begin{split}
\frac{|G|}{|C|}\sum_{  \mathfrak{p}\in \mathcal{P}(C),  N\mathfrak{p}\le x^5} (\log N\mathfrak{p}) \hat{k} (N\mathfrak{p})
\ge &  k(1)-k(\beta_1)  - \sum_{|\gamma|\le 1}|k(\rho )|
\\
& -19.17 x \Lo -1.8292\Lo 
 - 5.4568 x(\log x)^{\frac{1}{2}} n_L  \Lo,
  \end{split}
\end{equation}
where, as later, the sum on the right is over non-trivial zeros $\rho =\beta+i\gamma \neq \beta_1$ of $\zeta_L(s)$.\footnote{The values appearing in \eqref{L8.5} are respectively called $c_{20}, c_{15}', \alpha_4$ in \cite{AK}.}\\
We now emphasize how to improve the estimate for $  k(1)-k(\beta_1) $ as well as the contribution of the low-lying zeros of $\zeta_L(s)$. 
As $\beta_1$ is the closest to $1$, we use the mean value theorem to deduce
\[
k(1)-k(\beta_1) 
 \ge (1-\beta_1) (2\beta_1+1)  (\log x)x^{ \beta_1^2+\beta_1 }
 = 
c_{12}  (1-\beta_1)  \Lo  (2\beta_1+1) e^{2c_{12}\Lo}  e^{c_{12} (\beta_1^2+\beta_1-2)\Lo} .
 \]
Since 
\[
2\beta_1+1
\ge 
3 - \frac{2(\kappa C_1 )^2}{ \Lo_0^2} 
\ \text{and}\ 
 \beta_1^2+\beta_1 - 2
= (\beta_1+2)(\beta_1-1) 
\ge - \frac{ 3(\kappa C_1 )^2 }{\Lo^2} ,
\]
we have
\begin{equation}\label{diff1beta}
k(1)-k(\beta_1) 
 \ge 
\phi_0(c_{12} )  \Big( \Lo (1-\beta_1)  e^{2c_{12}\Lo} \Big) c_{12}, 
\end{equation}
where
\begin{equation}\label{def-phi_0}
\phi_0(t) =  e^{ - \frac{ 3 (\kappa C_1 )^2 }{\Lo_0} t} \Big( 3 -  \frac{2(\kappa C_1 )^2}{ \Lo_0^2} \Big)    .
\end{equation}
We split the sum over the zeros $\rho =\beta+i\gamma\neq \beta_1$, with $|\gamma |\le 1$, into $ \Sigma =  \Sigma_1+\Sigma_2$,
where
\[
\sideset{}{_1}\sum = 
\sum_{ \beta < 1/2}
+\frac12 \sum_{ \beta =  1/2} \ \text{and}\ 
\sideset{}{_2}\sum = 
\frac12 \sum_{ \beta =  1/2 } 
+\sum_{ \beta > 1/2  }.
\]
We use the bound
$
|k(\rho )|  \le x^{\beta^2 +\beta} \le x^{3/4}$
in the first part, together with the bound  $ N_L(1) \le c_{13}\Lo$, as given in \eqref{zero-density}, where 
\begin{equation}\label{c21}
c_{13} = \frac{1}{\pi} + 0.296 +  \mathscr{Q}_0 \Big(3.971 -\frac{\log(2\pi e)}{\pi} \Big)  +\frac{3.969 }{\Lo_0}.
\end{equation}
Thus, $\Sigma_1  
\le  \frac{N_L(1)}2 x^{\frac{3}{4}}
\le \frac{c_{13}}2 \Lo x^{\frac{3}{4}} $. Now, we assume $ c_{12}> \frac45 c_3$ and deduce
\begin{equation}\label{bnd-smallzeros1}
\sideset{}{_1}\sum
\le \frac{c_{13}}2  e^{(c_3-\frac{5}{4}c_{12})\Lo} 
\Big( (1-\beta_1) \Lo e^{2c_{12}\Lo}\Big)
\le \frac{c_{13}}2  e^{-(\frac{5}{4}c_{12}-c_3)\Lo_0} 
\Big( (1-\beta_1) \Lo e^{2c_{12}\Lo}\Big)
\end{equation}
since 
$\Lo x^{\frac{3}{4}} = \big(x^2 \Lo (1-\beta_1)\big) \frac{x^{-5/4}}{(1-\beta_1)}$,
where
$x=e^{c_{12}\Lo}$, and 
$1-\beta_1\ge d_L^{-c_3} = e^{-c_3\Lo}$ according to Theorem \ref{upperbd}. 
\\
For $\Sigma_2$, we apply $\beta^2 \le 2\beta=2-2(1-\beta)$ so that
\begin{equation}\label{S21}
|k(\rho )|   
 \le \frac{ x^{-2(1-\beta)} }{ (1-\beta_1)\Lo }\Big( (1-\beta_1) \Lo x^2\Big)
= \frac{ e^{-2c_{12}(1-\beta)\Lo} }{ (1-\beta_1)\Lo } \Big( (1-\beta_1) \Lo e^{2c_{12}\Lo}\Big).
\end{equation}
Together with Deuring-Heilbronn phenomenon, as described in Theorem \ref{thm-DH-all}:
\[(1-\beta)\Lo \ge C_2 \log   C_1  - C_2 \log\big( (1-\beta_1) \Lo \big) ,\]
we obtain
\begin{equation}\label{S22}
e^{-2c_{12}(1-\beta)\Lo }
 \le 
C_1^{-2 c_{12}C_2    }
( (1-\beta_1) \Lo )^{ 2 c_{12}C_2   }.
\end{equation}
Combining \eqref{S21} and \eqref{S22} with $(1- \beta_1)\Lo < (\kappa C_1 )^2 \Lo^{-1}$, we have
\begin{equation}\label{S23}
|k(\rho )|   
 \le \frac{  C_1^{2 c_{12}C_2  -2} 
\kappa^{ 4 c_{12}C_2   -2} }{ \Lo^{ 2 c_{12}C_2   -1}  }\Big( (1-\beta_1) \Lo e^{2c_{12}\Lo}\Big) 
\end{equation}
provided that $c_{12}> \frac{1}{2C_2}$.
Thus, we derive
\begin{equation}\label{bnd-smallzeros2}
\sideset{}{_2}\sum
\le
  \frac{c_{13}}2\kappa^2
\Big( \frac{  \kappa^2 C_1 }{ \Lo_0}\Big)^{2c_{12} C_2  -2} 
\Big( (1-\beta_1) \Lo e^{2c_{12}\Lo}\Big) ,
\end{equation}
as long as $c_{12} > \frac{1}{C_2 }$.
Combining \eqref{L8.5}, \eqref{diff1beta}, \eqref{bnd-smallzeros1}, and \eqref{bnd-smallzeros2} finally gives that for
\[ c_{12}>\max\Big( \frac1{  C_2  },  c_3  \Big),
\] 
\begin{multline}\label{solve-extremelysmall}
 \Big( (1-\beta_1) \Lo e^{2c_{12}\Lo}\Big) ^{-1}
 \frac{|G|}{|C|}\sum_{ \mathfrak{p}\in \mathcal{P}(C),  N\mathfrak{p}\le d_L^{5c_{12}}} (\log N\mathfrak{p}) \hat{k} (N\mathfrak{p})
\\
\ge
\phi_0(c_{12} )  c_{12} 
 - \frac{c_{13} \kappa^2  }2 \Big( \frac{  \kappa^2 C_1 }{ \Lo_0} \Big)^{2c_{12} C_2  -2} 
- \mathcal{E}_4(c_{12}),
\end{multline}
where
\begin{equation}\label{def-epsilon4}
\mathcal{E}_4(t) = 
\frac{c_{13}}2  e^{-(\frac{5}{4} t -c_3)\Lo_0} 
+  
\Big( 19.17 
 + 5.4568 \cQ_0 \Lo_0^{\frac{3}{2}}   t ^{1/2} 
 \Big) e^{-( t -c_3)\Lo_0} 
+ 1.8292  e^{-(2 t -c_3)\Lo_0}.
\end{equation}

\subsection{Numerical results}

We choose $\varepsilon_1, \sigma_1$, and $\eta$ to make $c_4$ as small as possible while  \eqref{solve-medium}, \eqref{solve-small}, \eqref{solve-verysmall}, and \eqref{solve-extremelysmall} are satisfied.
\\
Note that the choice for $\nu$ and $\eta$ are balanced in the inequalities \eqref{solve-small} and \eqref{solve-verysmall} as $\frac{ c_4 \Lo_0^{\nu-2} }{\mu}$ and $\frac{\mu}{c_4 \Lo_0^{\nu-1}}$ both have to be small enough.
Similarly, the choice of $\kappa$ is balanced by the inequalities \eqref{solve-verysmall} and \eqref{solve-extremelysmall} as both $\frac{ c_6 }{ (\kappa C_1 )^2}$ and $\frac{  \kappa^2 C_1 }{ \Lo_0}$ need to be small. 
The choice for $\varepsilon_2$ and $\sigma_2$ is to make $\frac{ c_6 }{ (\kappa C_1 )^2}$ as small as possible in \eqref{solve-verysmall}.
That choice was satisfactory to keep \eqref{solve-extremelysmall} valid. 
Finally, we choose $\theta$ in each case to make $B$ as small as possible. 
We recall that \[R_0=20\ \text{ and }\ R_1=1.24 .\]
We detail the case when $n_0=9$.
From Table \ref{Table-Fiori}, for $\Lo \le \log( 2.29 \cdot10^7 )$, we have 
\[B\le 1.7712.\]
In the case where $\Lo> \log( 2.29 \cdot10^7)$, we first investigate the case where there are no exceptional zero. Choosing $\theta=33.27$ and $\alpha=2.56$ and solving in $c_4$ so that \eqref{solving_c4} is satisfied then lead to \[B=42.3849.\]
In the case the exceptional zero $\beta_1$ exists, we set
\[\varepsilon_1=5.57, \ 
\sigma_1=4.45, \ 
\eta=0.025\]
so that
\[
c_1^{\prime}(\varepsilon_1,\sigma_1,\eta) = 0.002509182, 
\ \text{and}\ 
c_2^{\prime}(\varepsilon_1,\sigma_1,\eta) = 0.04890427.
\]
This gives
\[
c_8 = 0.003324331, 
\ \text{and}\ 
c_4 =  150.4072.
\]
We also take 
\[
\varepsilon_2 =5.97, \ 
\sigma_2 = 4.5, \ 
\kappa =23, \ 
\lambda = 0.2, \ 
\mu = 0.1 ,
\ \text{and}\  
\nu = 1.15,
\]
i.e., we work over the ranges:
\[
\frac{  0.003330581}{ \Lo}
<
\frac{0.0006648618}{\Lo^{0.15} } 
<
0.001329724
<
0.05 .
\]
We find in the ``extremely small'' case that
\[
c_3 = 9.85\ldots, \ c_{12} = 34.97\ldots , \text{and}\ B = 5 c_{12} = 174.8780. 
\]
Finally, we choose $\alpha$ and $\theta$ in each other case and obtain
\begin{center}
\begin{tabular}{| l | c | c | c | }
\hline 
Case      & Medium       & Small           &  Very small  \\ \hline 
$\theta$ & $1.02$         & $1.02$         & $1.029$           \\ \hline 
$\alpha$ & $5.85$        & $0.17$         & $0.67$           \\ \hline 
$B$       & $306.8307$  & $306.8307$ & $309.5380$   \\ \hline 
\end{tabular}
\end{center}
For all the remaining degrees $n_L\neq 9$, we then do the calculations using the same parameter values  except for $\varepsilon_2$ which we choose optimally for each degree.

\begin{center}
 \begin{table}[h]
   \caption{Table of admissible $(n_0,d_0,B)$}\label{table}
 \begin{tabular}{|c|c|c|c|c|c|}
\hline
 $n_0$ & 2 &  3 &  4 &  5 &  6 \\
 \hline
 $d_0$ &  $10^{10}$ & $10^{10}$ & $10^{8}$ & $10^{8}$ & $10^{8}$  \\
 \hline 
 $B$ & 223.2 & 231.7 & 249.1 & 259.8 & 271.1 \\
\hline 
\hline
 $n_0$ & 7 &  8 & 9 &  10 &  11  \\
 \hline
 $d_0$ & $10^{8}$ & $10^{7}$ &  $2.29\cdot 10^{7}$ & $1.56\cdot 10^{8}$ & $3.91\cdot 10^{9}$   \\  
\hline
 $B$ & 280.5  & 303.3 & 309.6 & 309.4 & 303.0  \\
 \hline
\hline
 $n_0$ &  12 &  13 &  14 &  15 & 16  \\
 \hline
 $d_0$ & $2.74\cdot 10^{10}$ & $7.56\cdot 10^{11}$ & $5.43\cdot 10^{12}$ & $1.61\cdot 10^{14}$ & $1.17\cdot 10^{15}$ \\  
 \hline
  $B$ & 303.2 & 298.4 & 298.8 & 295.1 & 295.6  \\ 
 \hline   
\hline
 $n_0$ &  17 &  18 &  19 &  20 &   $  21+$  \\
 \hline
 $d_0$ &  $3.70\cdot 10^{16}$ & $2.73\cdot 10^{17}$ & $9.03\cdot 10^{18}$ & $6.74\cdot 10^{19}$ & $10^{n_L} $  \\  
 \hline
  $B$ & 292.5 & 293.0 & 290.4 & 291.0 &   $ 290.2$ \\ 
 \hline   
 \end{tabular}
 \end{table}
\end{center}
We find all admissible values for $B$ are no larger than the above $309.5380$ (see Table \ref{table} for the values of $B$ for each case).
This proves that for any Galois extension $L/K$ of number fields with Galois group $G$, $C$ a conjugacy class in $G$, $n_L=n_0$ and $d_L\ge d_0$, there exists an unramified prime  $\mathfrak{p}$ of $K$, of degree one,  such that $\sigma_{\mathfrak{p}}=C$ and $N \mathfrak{p} \le d_{L}^{B}$ with 
admissible $(n_0,d_0,B)$ recorded in Table \ref{table}.
Together with Fiori's numerical verifications recorded in Table \ref{Table-Fiori}, this concludes the proof of Theorem \ref{mainthm}.

\subsection{A summary of key ideas}\label{keys}

We list here what we have done differently from \cite{AK}.\\

\noindent (i) For $d_L$ relatively small, the least prime is bounded numerically (see Appendix).

\noindent (ii) If the exceptional zeros exists, we have an enlarged region that contains no non-exceptional (Proposition \ref{lemma-zfr-enlarged}). (Cf. \cite[Eq. (6.1)]{AK}.)

\noindent (iii) We obtain a better version of the Deuring-Heilbronn phenomenon (cf. \cite[Theorem 7.3]{AK}) which is due to
\begin{itemize}[leftmargin=15pt,labelindent=0pt,labelwidth=12pt,itemindent=8pt,topsep=0pt]
 \item the refined Tur\'an's power sum method established in \cite{KNW} (cf. \cite[Theorem 4.2]{LMO}), and 
 \item better estimates for $\frac{\Gamma'}{\Gamma}$, i.e., Lemma \ref{bnd-gammas} and \eqref{bnd-gammas-t<1}, (cf. \cite[Lemmata 5.3 and 5.4]{AK}).
\end{itemize}
In addition, we do not follow \cite[Theorem 7.3]{AK}) to split our consideration into imaginary and non-imaginary cases but only consider the non-trivial zeros of $\zeta_L(s)$ (as the Deuring-Heilbronn phenomenon is ``trivial'' for the trivial zeros of $\zeta_L(s)$).

\noindent (iv) Regarding the weight $k$:
\begin{itemize}[leftmargin=15pt,labelindent=0pt,labelwidth=12pt,itemindent=8pt,topsep=0pt]
 \item we change the weight $k$ as in \eqref{def-k} for the  ``non-extremely-small'' cases.
  \item we use better estimates for $k(1)-k(\beta_1)$, namely, \eqref{bnd-ksigma<1} (together with Lemma \ref{phi-lemma}) and \eqref{diff1beta}.
\end{itemize}

\noindent (iv) Regarding the position of the exceptional zero $\beta_1$ of the Dedekind zeta function:
\begin{itemize}[leftmargin=15pt,labelindent=0pt,labelwidth=12pt,itemindent=8pt,topsep=0pt]
\item we consider a more refined splitting for the position of $\beta_1$ (Section \ref{small-exp-zero}; cf. \cite[Sec. 8]{AK}); our ``small'' case (when $(1-\beta_1)\Lo$ is smaller than constant size) is giving the worst $B$. Splitting this further with a ``very small'' case allowed further improvement. 
\end{itemize}

\noindent (v) Regarding the zeros of the Dedekind zeta function:
\begin{itemize}[leftmargin=15pt,labelindent=0pt,labelwidth=12pt,itemindent=8pt,topsep=0pt]
\item we use the improved zero-density estimates \eqref{zero-density}, established in \cite{HSW}, and \eqref{counting-zero-bd-fixed} (cf. \cite[Propositions 5.5 and 5.6]{AK}).
\item we obtain a better estimate for zeros $\rho$ with $|\Im (\rho)|\le 1$ (see Propositions \ref{prop-explicit-ineq-zeros} and \ref{prop-exceptional}; cf. \cite[Lemma 8.2]{AK}) by
     \begin{itemize}
        \item noticing that there are no non-exceptional zeros close to 1 (see \eqref{sum-non-beta1}), and
        \item balancing the use of zero-free region and zero repulsion in the comparison of (a) and (b) done in the proof of  Proposition \ref{prop-exceptional} (cf. \cite[(i) and (ii) of the proof of Lemma 8.2]{AK}).
     \end{itemize}
     \item a consequence of not over-counting  low-lying zeros is to improve the value of $B$ in the ``medium'' case.
\item we realize that for the ``extremely small'' case the contribution of the zeros $\rho$, with $\Re(\rho)\le \frac{1}{2}$, to $\sum_{\rho}|k(\rho)|$ is small (see \eqref{bnd-smallzeros1}; cf. \cite[Lemma 8.5]{AK}). By this observation, we essentially half the size of the sum. 
\item we use the stronger repulsion of zeros $\rho$, with $|\Im (\rho)|\le 1$, from Theorem \ref{thm-DH-all} instead of the general one as used in \cite[Theorem 7.3(2)]{AK}.
\end{itemize}

Lastly, as we used a different labeling for the constants from \cite{AK}, we provide here the following table for comparison:
 \begin{center}
\begin{tabular}{|c|c|c|c|c|c|c|c|c|c|c|c|c|}
\hline
 Ahn-Kwon's notation & $c_7$ &  $c_8$ &  $c_{10}$ &  $c_{11}$ &  $c_{12}$ &  $c_{13}$ &  $c_{14}$ &  $c_{15}$ &  $c_{16}$ &  $c_{21}$ &  $c_{23}$ &  $\check{c}$\\
 \hline
 our new notation &  $c_1$ & $c_2$ & $c_3$ & $c_9$ & $c_8=c_{10}$ & $c_6$ & $c_{11}$ & $c_5$ & $c_4$ & $c_{13}$ & $c_{12}$  & $8+\varepsilon$ \\  
 \hline 
\end{tabular} 
\end{center}
Note that in \cite{AK}, there is no constant corresponding to our $c_7$ as the non-exceptional case is not treated separately there.

\section{Lower bound for the Chebotarev density theorem}\label{lower-bound}

\begin{proof}[Proof of Theorem \ref{lwr-bnd}]
Throughout this section, we shall adapt the notation used in \cite{AK2}. We let $c_6=11.7$ and $1-\beta_1 \ge d_L^{-c_6}$. 
\footnote{$c_6$ in this section is our $c_3$ from Theorem \ref{upperbd}.} We shall  assume $\log x\ge d_L^{c_6}$. Let $a\in(1,2]$. We consider the kernel $k_a(s)$ defined by
$$
k_a(s)=\frac{(x^s-1)(a^s-1)}{s^2 \log a}.
$$
Recall that the inverse Mellin
transform  of $k_a(s)$ is
$$
 \widehat{k_a}(u) =\frac{1}{2\pi i}\int_{2-i\infty}^{2+i\infty}k_a(s) u^{-s}ds
 =\left.
  \begin{cases}
   \frac{\log u}{\log a} & \text{if } 1 < u < a , \\
   1 & \text{if } a \le u \le x , \\
   \frac{1}{\log a}\log (\frac{ax}{u}) & \text{if } x < u \le ax , \\
   0 & \text{if }   u > ax.\\
  \end{cases}
  \right.
$$
By \cite[Lemma 2.3]{AK2}, we know that
\begin{equation} \label{DC-1}
\mathcal{D}_C=\sum_{\mathfrak{p}\in \mathcal{P}(C)} (\log N\mathfrak{p}) \widehat{k_a} (N\mathfrak{p}) \le \tilde{\pi}_C(ax) \log x.
\end{equation}
Let $x_0 = 3^{c_6}$. Then for $x\ge \exp(x_0)$, one has
\begin{equation} \label{DC-2}
\frac{|G|}{|C|}\mathcal{D}_C\ge c_{43}(a)x,
\end{equation}
where
\[
c_{43}(a)= 0.49 \frac{a-1}{\log a} -c_{41} e^{- c_{39} \sqrt{\frac{x_0}{2}}}
- \Big(c_{35}x_0 \log x_0 +c_{40}  x_0\Big)  e^{- \frac{x_0}{2}}
\]
and $c_{35}$, $c_{39}$, $c_{40}$, and $c_{41}$ are constants, depending on $x_0$ and $a$, which are defined in \cite[pp. 304--306]{AK2}.
By \eqref{DC-1} and  \eqref{DC-2}, we have
 $$
\frac{|G|}{|C|} \tilde{\pi}_C(ax) \log x \ge c_{43}(a)x.
$$
Hence, replacing $x$  by $x/a$, together with a simplification, we deduce
$$
\tilde{\pi}_C(x) \ge \frac{c_{43}(a)}{a}  \frac{|C|}{|G|}\frac{x}{\log x -\log a} \ge  \frac{c_{43}(a)}{a}  \frac{|C|}{|G|}\frac{x}{\log x }.
$$
Finally, we conclude the proof by choosing $a$ as close to $1$ as possible.
Namely, for $a=1.0001$, we find $ m=\frac{c_{43}(a)}{a} =0.489975\ldots$.
\end{proof}

\begin{remark}
Note that, for $a\approx 1$,
$ c_{25}=2(50.313...)\frac{a+1}{\log a}$, 
$c_{35}\approx \frac{2}{c_6 \log 3}c_8$, 
$c_{39}\approx  \frac{1}{\sqrt{29.57}}$,  
$c_{40}\approx \sqrt{a} \approx 1$, and
$c_{41}\approx \frac{5.7868c_{39}}{c_6}\frac{a+1}{\log a}$.
Note that, as $\frac{a-1}{\log a}\approx 1$ when $a\approx 1$, then $ m=\frac{c_{43}(a)}{a} \approx 0.49$. 
This value $0.49$ comes from the choice of the kernel $k_a(s)$, so this may be improved by using a different kernel.
\\
Note that for $a=2$ we find $m=\frac{c_{43}(a)}{a}=0.353460\ldots$.
\end{remark}

\section*{Acknowledgments}
The authors are grateful to the referees for their careful reading of this article and for their helpful comments.

\appendix

\section{\large Numerical Verification of the Least Prime in the Chebotarev Density Theorem - by Andrew Fiori\footnote{The majority of computations for this work were provided on systems supported by Compute Canada. Additional hardware at the University of Lethbridge was purchased through an NSERC RTI grant and supported by funding from the University of Lethbridge.} }
 \label{Appendix}

\subsection{Introduction}

The purpose of this appendix is to document partial numerical verification of different bounds for the least prime in the Chebotarev density theorem.
There are several known asymptotic bounds for the least prime \cite{BS,LMO,Z17,AK,KNW}, the shape of these bounds depends on whether one assumes the GRH. 
For each of these bound shapes we document the worst case behaviour for fields up to some bound on the discriminant and the Galois type of the field.
Here Galois type refers to the Galois closure of the field.

The numerical verification documented here involved two essential steps:
\begin{enumerate}
 \item Obtaining proven complete tabulations of number fields up to some discriminant bound. We describe our methods in Section \ref{sec:tab}.
 \item For each field, and each automorphism of that field, finding the least prime in the Chebotarev density theorem. We describe our methods in Section \ref{sec:ver}.
\end{enumerate}
A summary of our results by degree is in Section \ref{sec:data}. More complete tables of results are available from \cite{NumericalChebotarevWebsite}.
Each step of this process made extensive use of the PARI-GP library \cite{PARI2}

\subsection{Tabulation of Number Fields}\label{sec:tab}

Many researchers have made contributions towards the development of tabulations of number fields with small discriminant. In particular,
we refer the reader to the following (incomplete) list of references:
\cite{JonesRoberts,KlunersMalle,simon1998,HunterQuinticFields,PohstMinimal,MartinetMinimal,BelabasCubic,QuarticKummer,
TableOfQuintic,PrimitiveSextic,SexticOverQuad,SexticOverCube,DiazDeg7,DiazDeg8,TROctic,SelmaneDeg8,ImprimDeg10,SelmaneDeg10A,JonesRamSearch,VoightTRFields}.
Moreover, several efforts have been made to construct reference databases particularly by the PARI group, Kluners and Malle, and Jones and Roberts.

The above work and databases form the basis for our own tabulations which in some cases extend these.
The tabulations we have made are available in \cite{FieldTabulationWebsite}. Note that extending these tabulations is an ongoing project.

\subsubsection{Methods for Tabulations of Number Fields}

For degree two fields tabulations are essentially trivial, and for degree three fields the work of Belabas \cite{BelabasCubic} provides an efficient method to obtain tabulations.
For higher degree fields there are several strategies which may be applicable depending on the type of the field.
\begin{itemize}
 \item For large degrees the Minkowski bounds provide strong lower bounds for the minimal root discriminant of a field, see for example \cite{DiazMinkowskiDiscriminantBounds}
       We used this as the basis for completion results for degrees 9 and higher. Note that for a field of degree $n_L$ which is $21$ or larger we have $d_L > 10^{n_L}$.
       
 \item If the type of the field being considered is primitive then a search based on the geometry of numbers using ideas from \cite{HunterQuinticFields, PohstMinimal} provides a method
       to obtain a complete tabulation up to a chosen bound on the discriminant.     
       We used this method to extend the tabulations for relevant field types and in particular to improve upon existing tabulations in degrees $4,5,6$, and $8$.
       
 \item If the type of the field being tabulated admits an automorphism one may use Kummer theory to build all possible relative extensions based on complete tabulations
       of the relevant lower degree field. PARI-GP includes complete implementations of the aspects of class field theory and Kummer theory necessary to implement these methods.  
       We used this method to extend the tabulations for relevant field types in particular to improve the completion in degree $4,6,8$.
       
 \item If the type of field being tabulated admits a subfield, over which the relative Galois group (of the relative Galois closure) is solvable, one can use
          Kummer theory in a more elaborate way building off the ideas used in \cite{QuarticKummer}. Such strategies also work if the Galois group over $\mathbb{Q}$ is solvable.        
          PARI-GP includes complete implementations of the aspects of class field theory and Kummer theory necessary to implement these methods.        
          We used this method to extend the tabulations for relevant field types in particular to improve the completion in degree $6,9$.
          
 \item If the field being tabulated admits a subfield then the geometry of numbers methods of \cite{MartinetMinimal} may be used to build all relative extensions.
          This method may be used even when Kummer theory can be used. Kummer theory will typically be simpler.        
          We have not made use these methods in our search as none of the extension degrees for which we can obtain results would benefit.
\end{itemize}

\subsection{Verification of Least Prime}\label{sec:ver}

We now document our method of verification.

Suppose $L/\mathbb{Q}$ is a fixed field with absolute discriminant $d_L$.
Let $G = {\rm Aut}_{\mathbb{Q}}(L)$ and $\sigma\in G$.

What we verify is that there exists a \textit{small} prime $p$ of $\mathbb{Q}$ which is unramified in $L$, for which there exists a prime $\mathcal{P}$ of $L$ over $p$ such that
\begin{enumerate}
 \item $\sigma(\mathcal{P}) = \mathcal{P}$.
 \item $\forall x\in\mathcal{O}_L, \sigma(x) \equiv x^p \pmod{\mathcal{P}}$,
\end{enumerate}
where small means relative to one of the types of bound under consideration relative to the discriminant $d_{L}$ of $L$.

We note that this formulation differs from how one typically states the conditions on the least prime in the Chebotarev density theorem.
We state the following claims without proof:

\begin{claim}
 If $K=L^{\sigma}$ is the fixed field of $\sigma$ then $\mathfrak{p} = \mathcal{P}\cap K$ satisfies the usual conditions of the least prime in the Chebotarev density theorem if and only if $\mathcal{P}$ satisfies these 
 slightly different conditions.
  Hence, we may find the smallest value of $N_{K/\mathbb{Q}}(\mathfrak{p}) = p$ by finding the smallest value of $p$.
\end{claim}

\begin{claim}
 If $K \subset L^{\sigma}$ is the fixed field of $\sigma$ then if $\mathcal{P}$ satisfies these 
 slightly different conditions then 
 $\mathfrak{p} = \mathcal{P}\cap K$ satisfies the conditions of the least prime in the Chebotarev density theorem with respect to the image of $\sigma$ in ${\rm Aut}_K(L)$.
  Hence, we may find an upper bound for the smallest value of $N_{K/\mathbb{Q}}(\mathfrak{p}) = p$ by finding the smallest value of $p$.
\end{claim}

The above two claims are useful because they reduce the total number of cases that needed to be checked, but also because they make clear that an explicit check for any given field
and any given automorphism is actually computationally practical.
Indeed, PARI-GP includes all the tools necessary to make such a check, and hence find the smallest $p$.
\begin{rmk}
 Because $p$ is expected to be small, the runtime of this search is typically very small.
 
 We remark also that PARI-GP has all the tools necessary to compute ${\rm Aut}_{\mathbb{Q}}(L)$.
\end{rmk}

\subsection{Summary of Results}\label{sec:data}

Table \ref{Table-Fiori} summarizes the results we have obtained.

More complete tables, as well as searchable lists of all results for all fields referenced can be found at \cite{NumericalChebotarevWebsite}.

There are essentially three types of bounds considered

\begin{enumerate}

\item[A] bound $d_L^A$ which is based on the GRH-unconditional results.

\item[B] bound $B(\log d_L)^2$ which is based on the GRH-conditional results.

\item[C] bound $C(3e^\gamma/\pi)^2 \frac{ (\log d_L)^2 (\log(2\log\log d_L))^2}{(\log\log d_L)^2}$ which is based on an upper bound for the worst known family of fields \cite{Fiori10}.

\end{enumerate}

For each row of the table and each type of bound, there are no fields $L$ with $n_L=n_0$ and $d_L \leq d_0$, for which the least prime in the Chebotarev density theorem exceeds the bounds.
Additionally, for each row, and each type of bound, there exists a field $L$, with $n_L=n_0$, which (up to the precision given) realizes the given bound.
In many cases these worst case fields will have absolute discriminants larger than the verification height.
One can thus interpret each row as giving a \textit{lower bound} on the upper bound for the least prime in the Chebotarev theorem.
The column $d_{min}$ documents a lower bound for the absolute discriminant of fields of that degree, so that for all $L$ with $n_L=n_0$ we have $d_L \ge d_{min}$.
We recall that for degrees 9+ there are in fact no fields whatsoever with $n_L=n_0$ and $d_L \leq d_0$.
For degree $21$ and higher this verification height, and $d_{min}$ can be taken as at least $10^{[L:Q]}$.

\begin{small}

\begin{table}[h]

\caption{Table of worst case bounds}

\label{Table-Fiori}

\begin{tabular}{| c | c | c | p{1.5cm} | p{1.5cm} | p{1.5cm} | }

\hline

$n_0=[L:\mathbb{Q}]$ & $d_{min}$ & Verification Height $d_0$ & A & B & C\\

\hline

2 &$ 3 $& $16$ & 1.7712 &5.7997 &136.0600 \\

2 & & $16 < d_L \leq 10^{10}$ & 0.8071 & 1.5802 & 1.0389 \\

3 &$ 23$ & $10^{10}$ & 0.6591& 0.8803 &0.5661 \\

4 &$ 117$ & $10^8$ & 0.6098 & 0.9979 & 0.6823\\

5 &$ 1609$ & $10^8$ & 0.3559 & 0.6668 & 0.5297\\

6 &$ 9747$ & $10^8$ & 0.4335 & 0.6668 & 0.5297 \\

7 &$ 184607$ & $10^8$ & 0.2490 & 0.2050 & 0.1890 \\

8 & $1257728$& $10^7$ & 0.2925 & 0.5047 & 0.5327 \\

9 & $2.29\cdot 10^7 $& $2.29\cdot 10^7$ & 0.2183 & 0.1692 & 0.1658\\

10 & $1.56\cdot 10^8$& $1.56\cdot 10^8$ & 0.2295 & 0.4319 & 0.5069\\

11 &$3.91\cdot 10^9$ & $3.91\cdot 10^9$ & 0.1228 & 0.1047 & 0.1272 \\

12 & $2.74\cdot 10^{10}$& $2.74\cdot 10^{10}$ & 0.2107 & 0.3269 & 0.3393\\

13 & $7.56\cdot 10^{11}$& $7.56\cdot 10^{11}$ & 0.1365 & 0.0627 & 0.0690 \\

14 & $5.43\cdot 10^{12} $& $5.43\cdot 10^{12} $ & 0.1582 & 0.4006 & 0.5268 \\

15 & $1.61\cdot 10^{14}$ & $1.61\cdot 10^{14}$ & 0.1206 & 0.0866 & 0.1201 \\

16 &$1.17\cdot 10^{15}$& $1.17\cdot 10^{15}$ & 0.1463 & 0.2769 & 0.3410 \\

17 & $3.70\cdot 10^{16}$& $3.70\cdot 10^{16}$ & 0.1069 & 0.0478 & 0.0689 \\

18 & $2.73\cdot 10^{17}$& $2.73\cdot 10^{17}$ & 0.1279 & 0.3371 & 0.4778 \\

19 &$9.03\cdot 10^{18}$& $9.03\cdot 10^{18}$ & 0.0909 & 0.0327 & 0.0484\\

20+ &$6.74\cdot 10^{19} $ &$6.74\cdot 10^{19} $ & 0.1230 & 0.3370 & 0.4741 \\

\hline

\end{tabular}

\end{table}

\end{small}

\newpage

\providecommand{\MR}[1]{}
\providecommand{\bysame}{\leavevmode\hbox to3em{\hrulefill}\thinspace}
\providecommand{\MRhref}[2]{%
  \href{http://www.ams.org/mathscinet-getitem?mr=#1}{#2}
}
\providecommand{\href}[2]{#2}


\begin{thebibliography}{99}

\bibitem[AhKw14]{AK14} J.-H. Ahn and S.-H. Kwon, {\it Some explicit zero-free regions for Hecke L-functions}, J. Number Theory  \textbf{145} (2014), 433--473.

\bibitem[AhKw19-1]{AK}
\bysame, {\it An explicit upper bound for the least prime ideal in the Chebotarev density theorem},  Ann. Inst. Fourier (Grenoble)  \textbf{69} (2019), 1411--1458.

\bibitem[AhKw19-2]{AK2}
\bysame, {\it Lower estimates for the prime ideal of degree one counting function in the Chebotarev density theorem}, Acta Arith.  \textbf{191} (2019), no. 3, 289--307. 

\bibitem[BaSo96]{BS}
E. Bach and J. Sorenson, {\it Explicit bounds for primes in residue classes},  Math. Comp.  \textbf{65} (1996), no. 216, 1717--1735.

\bibitem[Bel97]{BelabasCubic}
K.~Belabas, \emph{A fast algorithm to compute cubic fields}, Math. Comp. \textbf{66} (1997), no.~219, 1213--1237. 

\bibitem[BeMaOl90]{SexticOverQuad}
A.-M. Berg\'{e}, J.~Martinet, and M.~Olivier, \emph{The computation of sextic fields with a quadratic subfield}, Math. Comp. \textbf{54} (1990), no.~190, 869--884. 

\bibitem[CoDyDOl03]{QuarticKummer}
H. Cohen, F. Diaz~y Diaz, and M. Olivier, \emph{Constructing complete tables of quartic fields using {K}ummer theory}, Math. Comp. \textbf{72} (2003), no.~242, 941--951. 

\bibitem[Das21]{DasMSc}
S. Das, {\it An explicit version of Chebotarev's density theorem}
MSc Thesis, Dec. 2020, \url{https://opus.uleth.ca/bitstream/handle/10133/5825/DAS_SOURABHASHIS_MSC_2020.pdf?sequence=3&isAllowed=y}

\bibitem[Deu35]{De} 
M. Deuring, {\it\"Uber den Tschebotareffschen Dichtigkeitssatz}, Math. Ann.
 \textbf{110} (1) (1935), 414--415.

\bibitem[DyD80]{DiazMinkowskiDiscriminantBounds}
F. Diaz~y Diaz, \emph{Tables minorant la racine {$n$}-i\`eme du
  discriminant d'un corps de degr\'{e} {$n$}}, Publications Math\'{e}matiques
  d'Orsay 80 [Mathematical Publications of Orsay 80], vol.~6, Universit\'{e} de
  Paris-Sud, D\'{e}partement de Math\'{e}matique, Orsay, 1980. 

\bibitem[DyD84]{DiazDeg7}
\bysame, \emph{Valeurs minima du discriminant pour certains types de corps de
  degr\'{e} {$7$}}, Ann. Inst. Fourier (Grenoble) \textbf{34} (1984), no.~3,
  29--38. 

\bibitem[DyD87]{DiazDeg8}
\bysame, \emph{Petits discriminants des corps de nombres totalement
  imaginaires de degr\'{e} {$8$}}, J. Number Theory \textbf{25} (1987), no.~1,
  34--52. 

\bibitem[DrJo09]{JonesRamSearch}
E.~D. Driver and J.~W. Jones, \emph{A targeted {M}artinet search}, Math.
  Comp. \textbf{78} (2009), no.~266, 1109--1117. 

\bibitem[DJ10]{ImprimDeg10}
\bysame, \emph{Minimum discriminants of imprimitive decic fields}, Experiment.
  Math. \textbf{19} (2010), no.~4, 475--479. 

\bibitem[Fio18]{Fiori10}
A. Fiori, \emph{Lower bounds for the least prime in {C}hebotarev},  Algebra Number Theory \textbf{13} (2019), no. 9, 2199--2203. 

\bibitem[FioWeb1]{NumericalChebotarevWebsite}
\bysame, \emph{The first prime in Chebotarev}, 2019. \\http://www.cs.uleth.ca/~fiori/ResearchCode/ExplicitChebotarev/

\bibitem[FioWeb2]{FieldTabulationWebsite}
\bysame, \emph{Tabulations of number fields of bounded discriminant by Galois type}, 2019. \\http://www.cs.uleth.ca/~fiori/ResearchCode/LowDiscFields/

\bibitem[HaShWo]{HSW} 
E. Hasanalizade, Q. Shen, and P.-J. Wong {\it Counting Zeros of Dedekind Zeta Functions}, arXiv:2102.04663.

\bibitem[Hun57]{HunterQuinticFields}
J. Hunter, \emph{The minimum discriminants of quintic fields}, Proc. Glasgow
  Math. Assoc. \textbf{3} (1957), 57--67. 

\bibitem[JoRo14]{JonesRoberts}
J.~W. Jones and D.~P. Roberts, \emph{A database of number fields}, LMS J.
  Comput. Math. \textbf{17} (2014), no.~1, 595--618. 

\bibitem[Kad12]{Ka12} 
H. Kadiri, {\it  Explicit zero-free regions for Dedekind zeta functions},  Int. J. Number Theory  \textbf{8} (2012), no. 1, 125--147.

\bibitem[KaNgWo19]{KNW} 
H. Kadiri, N. Ng, and P.-J. Wong, {\it The least prime ideal in the Chebotarev density theorem}, Proceedings of the American Mathematical Society  \textbf{147} (2019), 2289--2303. 

\bibitem[KlMa01]{KlunersMalle}
J. Kl\"{u}ners and G. Malle, \emph{A database for field extensions
  of the rationals}, LMS J. Comput. Math. \textbf{4} (2001), 182--196.

\bibitem[LaMoOd79]{LMO}
J. C. Lagarias, H. L. Montgomery, and A. M. Odlyzko,
{\it A bound for the least prime ideal in the Chebotarev density theorem}, Invent. Math.  \textbf{54} (1979), 271--296.

\bibitem[LaOd77]{LO}
J. C. Lagarias and A. M. Odlyzko, {\it Effective versions of the Chebotarev density theorem}, Algebraic number fields: L-functions and Galois properties (Proc. Sympos., Univ. Durham, Durham, 1975), pp. 409--464. Academic Press, London, 1977.

\bibitem[Lee]{Lee20}
E.~S. Lee, {\it On an explicit zero-free region for the Dedekind zeta-function}, J. Number Theory
 \textbf{224} (2021), 307--322.

\bibitem[LMFDB]{LMFDB}
The LMFDB Collaboration, {\it The L-functions and Modular Forms Database}, available from \url{http://www.lmfdb.org}. 

\bibitem[Mar85]{MartinetMinimal}
J. Martinet, \emph{Methodes g\'{e}om\'{e}triques dans la recherche des
  petits discriminants}, S\'{e}minaire de th\'{e}orie des nombres, {P}aris
  1983--84, Progr. Math., vol.~59, Birkh\"{a}user Boston, Boston, MA, 1985,
  pp.~147--179. 

\bibitem[Mon94]{Mon10}
H. L. Montgomery,
\emph{Ten lectures on the interface between analytic number theory and harmonic analysis},
CBMS Regional Conference Series in Mathematics, 
vol. 84, 1994. 
		
\bibitem[Oli90]{PrimitiveSextic}
M. Olivier, \emph{Corps sextiques primitifs}, Ann. Inst. Fourier (Grenoble)
  \textbf{40} (1990), no.~4, 757--767. 

\bibitem[Oli92]{SexticOverCube}
\bysame, \emph{The computation of sextic fields with a cubic subfield and no
  quadratic subfield}, Math. Comp. \textbf{58} (1992), no.~197, 419--432.

\bibitem[PARI]{PARI2}
{The PARI~Group}, Univ. Bordeaux, \emph{{PARI/GP version \texttt{2.11.2}}},
  2019, available from \url{http://pari.math.u-bordeaux.fr/}.

\bibitem[Pla16]{Pla16} D. J. Platt, Numerical computations concerning the GRH, Math. Comp. \textbf{85} (2016), no. 302, 3009--3027.


\bibitem[Poh82]{PohstMinimal}
M.~Pohst, \emph{On the computation of number fields of small discriminants
  including the minimum discriminants of sixth degree fields}, J. Number Theory
  \textbf{14} (1982), no.~1, 99--117. 

\bibitem[PoMaDyD90]{TROctic}
M.~Pohst, J.~Martinet, and F.~Diaz~y Diaz, \emph{The minimum discriminant of
  totally real octic fields}, J. Number Theory \textbf{36} (1990), no.~2,
  145--159. 

\bibitem[ScPoDyD94]{TableOfQuintic}
A.~Schwarz, M.~Pohst, and F.~Diaz~y Diaz, \emph{A table of quintic number
  fields}, Math. Comp. \textbf{63} (1994), no.~207, 361--376. 

\bibitem[Sel99]{SelmaneDeg8}
S. Selmane, \emph{Non-primitive number fields of degree eight and of
  signature {$(2,3)$}, {$(4,2)$} and {$(6,1)$} with small discriminant}, Math.
  Comp. \textbf{68} (1999), no.~225, 333--344. 

\bibitem[Sel01]{SelmaneDeg10A}
\bysame, \emph{Quadratic extensions of totally real quintic fields}, Math.
  Comp. \textbf{70} (2001), no.~234, 837--843. 

\bibitem[Sim98]{simon1998}
D. Simon, \emph{Equations dans les corps de nombres et discriminants
  minimaux}, Ph.D. thesis, 1998, Th\`ese de doctorat dirig\'ee par H. Cohen, 
  Math\'ematiques pures Bordeaux 1 1998, p.~152 p.

\bibitem[ThZa17]{TZ17} 
J. Thorner and A. Zaman, {\it An explicit bound for the least prime ideal in the
Chebotarev density theorem}, Algebra Number Theory  \textbf{11} (2017), 1135--1197.

\bibitem[Tol97]{T0l97} 
E. Tollis, {\it Zeros of Dedekind Zeta Functions in the Critical Strip}, Math.
  Comp. \textbf{66} (1997), 1295--1321. 


\bibitem[Voi08]{VoightTRFields}
J. Voight, \emph{Enumeration of totally real number fields of bounded root
  discriminant}, Algorithmic number theory, Lecture Notes in Comput. Sci., vol.
  5011, Springer, Berlin, 2008, pp.~268--281. 

\bibitem[Win13]{Wi} B. Winckler, {\it  Th\'eor\`eme de Chebotarev effectif}, arXiv:1311.5715.

\bibitem[Zam17]{Z17}
A. Zaman, {\it Bounding the least prime ideal in the Chebotarev density theorem}, 
Funct. Approx. Comment. Math.  \textbf{57} (2017), no. 1, 115--142. 

\end{thebibliography}
\end{document}